\documentclass[10pt,twoside]{amsart}  
\def\YEAR{\year}\newcount\VOL\VOL=\YEAR\advance\VOL by-1995
\def\firstpage{1}\def\lastpage{1000}
\def\received{}\def\revised{}
\def\communicated{}

\makeatletter
\def\magnification{\afterassignment\m@g\count@}
\def\m@g{\mag=\count@\hsize6.5truein\vsize8.9truein\dimen\footins8truein}
\makeatother

\font\eightrm=cmr8
\font\caps=cmcsc10                    
\font\Caps=cmcsc10 scaled \magstep1   

\makeatletter
\@specialpagefalse\headheight=8.5pt
\def\DocMath{}
\renewcommand{\@evenhead}{%
    \ifnum\thepage>\lastpage\rlap{\thepage}\hfill%
    \else\rlap{\thepage}\slshape\leftmark\hfill{\caps\SAuthor}\hfill\fi}%
\renewcommand{\@oddhead}{%
    \ifnum\thepage=\firstpage{\DocMath\hfill\llap{\thepage}}%
    \else{\slshape\rightmark}\hfill{\caps\STitle}\hfill\llap{\thepage}\fi}%
\makeatother

\def\TSkip{\bigskip}
\newbox\TheTitle{\obeylines\gdef\GetTitle #1
\ShortTitle  #2
\SubTitle    #3
\Author      #4
\ShortAuthor #5
\EndTitle
{\setbox\TheTitle=\vbox{\baselineskip=20pt\let\par=\cr\obeylines%
\halign{\centerline{\Caps##}\cr\noalign{\medskip}\cr#1\cr}}%
	\copy\TheTitle\TSkip\TSkip%
\def\next{#2}\ifx\next\empty\gdef\STitle{#1}\else\gdef\STitle{#2}\fi%
\def\next{#3}\ifx\next\empty%
    \else\setbox\TheTitle=\vbox{\baselineskip=20pt\let\par=\cr\obeylines%
    \halign{\centerline{\caps##} #3\cr}}\copy\TheTitle\TSkip\TSkip\fi%
\centerline{\caps #4}\TSkip\TSkip%
\def\next{#5}\ifx\next\empty\gdef\SAuthor{#4}\else\gdef\SAuthor{#5}\fi%
\ifx\received\empty\relax
    \else\centerline{\eightrm Received: \received}\fi%
\ifx\revised\empty\TSkip%
    \else\centerline{\eightrm Revised: \revised}\TSkip\fi%
\ifx\communicated\empty\relax
    \else\centerline{\eightrm Communicated by \communicated}\fi\TSkip\TSkip%
\catcode'015=5}}\def\Title{\obeylines\GetTitle}
\def\Abstract{\begingroup\narrower
    \parskip=\medskipamount\parindent=0pt{\caps Abstract. }}
\def\EndAbstract{\par\endgroup\TSkip}

\long\def\MSC#1\EndMSC{\def\arg{#1}\ifx\arg\empty\relax\else
     {\par\narrower\noindent%
     2000 Mathematics Subject Classification: #1\par}\fi}

\long\def\KEY#1\EndKEY{\def\arg{#1}\ifx\arg\empty\relax\else
	{\par\narrower\noindent Keywords and Phrases: #1\par}\fi\TSkip}

\newbox\TheAdd\def\Addresses{\vfill\copy\TheAdd\vfill
    \ifodd\number\lastpage\vfill\eject\phantom{.}\vfill\eject\fi}
{\obeylines\gdef\GetAddress #1
\Address #2 
\Address #3
\Address #4
\EndAddress
{\def\xs{6truecm}\parindent=0pt 
\setbox0=\vtop{{\obeylines\hsize=\xs#1\par}}\def\next{#2}
\ifx\next\empty 
     \setbox\TheAdd=\hbox to\hsize{\hfill\copy0\hfill}
\else\setbox1=\vtop{{\obeylines\hsize=\xs#2\par}}\def\next{#3}
\ifx\next\empty 
     \setbox\TheAdd=\hbox to\hsize{\hfill\copy0\hfill\copy1\hfill}
\else\setbox2=\vtop{{\obeylines\hsize=\xs#3\par}}\def\next{#4}
\ifx\next\empty\ 
     \setbox\TheAdd=\vtop{\hbox to\hsize{\hfill\copy0\hfill\copy1\hfill}
                \vskip20pt\hbox to\hsize{\hfill\copy2\hfill}}
\else\setbox3=\vtop{{\obeylines\hsize=\xs#4\par}}
     \setbox\TheAdd=\vtop{\hbox to\hsize{\hfill\copy0\hfill\copy1\hfill}
	        \vskip20pt\hbox to\hsize{\hfill\copy2\hfill\copy3\hfill}}
\fi\fi\fi\catcode'015=5}}\gdef\Address{\obeylines\GetAddress}

\usepackage{amssymb}
\usepackage{hyperref}
\usepackage[capitalize]{cleveref}
\usepackage{stmaryrd}

\numberwithin{equation}{section}
\newtheorem{thm}[equation]{Theorem}
\newtheorem{cor}[equation]{Corollary}
\newtheorem{lem}[equation]{Lemma}
\newtheorem*{lem*}{Lemma}
\newtheorem{prop}[equation]{Proposition}
\theoremstyle{definition}
\newtheorem{rem}[equation]{Remark}
\newtheorem{notation}[equation]{Notation}
\newtheorem{defn}[equation]{Definition}
\newtheorem{assump}[equation]{Assumption}
\newtheorem{eg}[equation]{Example}
\newtheorem*{ack}{Acknowledgements}

\crefname{thm}{theorem}{theorems}
\Crefname{thm}{Theorem}{Theorems}
\crefname{cor}{corollary}{corollaries}
\Crefname{cor}{Corollary}{Corollaries}
\crefname{lem}{lemma}{lemmas}
\Crefname{lem}{Lemma}{Lemmas}
\crefname{prop}{proposition}{propositions}
\Crefname{prop}{Proposition}{Propositions}
\crefname{rem}{remark}{remarks}
\Crefname{rem}{Remark}{Remarks}

\crefformat{thm}{Theorem~#2#1#3}
\crefformat{cor}{Corollary~#2#1#3}
\crefformat{lem}{Lemma~#2#1#3}
\crefformat{prop}{Proposition~#2#1#3}
\crefformat{rem}{Remark~#2#1#3}
\crefmultiformat{cor}{Corollaries~#2#1#3}{ and~#2#1#3}{, #2#1#3}{, and~#2#1#3}
\crefmultiformat{thm}{Theorems~#2#1#3}{ and~#2#1#3}{, #2#1#3}{, and~#2#1#3}
\crefmultiformat{lem}{Lemmas~#2#1#3}{ and~#2#1#3}{, #2#1#3}{, and~#2#1#3}
\crefmultiformat{rem}{Remarks~#2#1#3}{ and~#2#1#3}{, #2#1#3}{, and~#2#1#3}
\crefmultiformat{defn}{Definitions~#2#1#3}{ and~#2#1#3}{, #2#1#3}{, and~#2#1#3}

 \newcounter{case}
\newenvironment{case}{\refstepcounter{case}
 \medskip \noindent {\bf Case \textup{\thecase.  }}\it}{\upshape}
 \renewcommand{\thecase}{\arabic{case}}
\crefformat{case}{Case~#2#1#3}

\DeclareMathOperator{\Ad}{Ad}
\DeclareMathOperator{\Aut}{Aut}
\DeclareMathOperator{\GL}{GL}
\DeclareMathOperator{\Int}{Int}
\DeclareMathOperator{\rank}{rank}
\DeclareMathOperator{\SL}{SL}
\DeclareMathOperator{\SO}{SO}
\newcommand{\ints}{\mathcal{O}}
\newcommand{\dual}{\widehat}
\newcommand{\unitary}{\mathcal{U}}
\newcommand{\hilbert}{\mathcal{H}}
\newcommand{\conj}{\overline}
\newcommand{\cl}[1]{\mathrm{cl}\!\left(#1\right)}
\newcommand{\CC}{\mathbb{C}}
\newcommand{\NN}{\mathbb{N}}
\newcommand{\PP}{\mathbb{P}}
\newcommand{\RR}{\mathbb{R}}

\newcommand{\ZZ}{\mathbb{Z}}
\newcommand{\Lie}[1]{\mathfrak{#1}}

\usepackage{textcomp}

\newcommand{\comm}[2]{\llbracket#1,#2\rrbracket}

\newcommand{\fullcref}[2]{\cref{#1}\pref{#1-#2}}
\newcommand{\fullCref}[2]{\Cref{#1}\pref{#1-#2}}
\newcommand{\pref}[1]{\upshape(\ref{#1}\upshape)}
\newcommand{\csee}[1]{\upshape(see \cref{#1}\upshape)}
\newcommand{\cseebelow}[1]{\upshape(see \cref{#1} below\upshape)}
\newcommand{\ccf}[1]{\upshape(cf.~\cref{#1}\upshape)}

\newcommand{\relT}{relative Property~\textup{(}T\textup{)}}

\makeatletter
\newcommand{\noprelistbreak}{\smallskip\@nobreaktrue\nopagebreak} 
\makeatother

\usepackage{color}
\newcommand{\refnote}[2][0]{\marginpar{%
	\color{blue}
	\vbox to 0pt{\vss
	$\begin{pmatrix} \text{see} \\[-3pt] \text{note} \\[-3pt] \text{\ref{#2}} \end{pmatrix}$%
	\vskip -1.1\baselineskip\vskip-#1pt}}}
\theoremstyle{definition}
\newtheorem{aid}{}
\numberwithin{aid}{section}
\newcommand{\oldaid}{}
\let\oldaid=\aid
\renewcommand{\aid}{\vfill\oldaid}
\newcommand{\oldendaid}{}
\let\oldendaid=\endaid
\renewcommand{\endaid}{\oldendaid\vfill\hrule width\textwidth \filbreak\bigskip}
\crefformat{aid}{Note~#2#1#3}

\begin{document}
\Title Relative Property (T) for Nilpotent Subgroups
{21 June 2017}
\ShortTitle Relative Property (T) for Nilpotent Subgroups
\SubTitle  
\Author Indira Chatterji, Dave Witte Morris, and Riddhi Shah
\ShortAuthor I.~Chatterji, D.~W.~Morris, and R.~Shah
\EndTitle
\Abstract 
We show that relative Property (T) for the abelianization of a nilpotent normal subgroup implies relative Property (T) for the subgroup itself. This and other results are a consequence of a theorem of independent interest, which states that if $H$ is a closed subgroup of a locally compact group~$G$, and $A$ is a closed subgroup of the center of~$H$, such that $A$ is normal in~$G$, and $(G/A, H/A)$ has relative Property (T), then $(G, H^{(1)})$ has relative Property (T), where $H^{(1)}$ is the closure of the commutator subgroup of~$H$. 
In fact, the assumption that $A$ is in the center of~$H$ can be replaced with the weaker assumption that $A$~is abelian and every $H$-invariant finite measure on the unitary dual of~$A$ is supported on the set of fixed points.
\EndAbstract
\MSC 
22D10.
\EndMSC
\KEY 
relative Property (T), 
nilpotent subgroup, almost-invariant vector,  fibered tensor product.
\EndKEY
\Address Laboratoire de Math\'ematiques 
	\qquad J.A. Dieudonn\'e
UMR no.~7351 CNRS UNS
Universit\'e de Nice-Sophia Antipolis
06108 Nice Cedex 02, France
indira.chatterji@math.cnrs.fr
\Address Department of Math and Comp Sci
University of Lethbridge
Lethbridge, Alberta T1K~3M4, Canada
dave.morris@uleth.ca
\Address School of Physical Sciences,
 Jawaharlal Nehru University
  New Delhi 110 067, India
  riddhi.kausti@gmail.com
\Address
\EndAddress

\thispagestyle{empty} 
\makeatletter\setlength\parindent {20\p@}\makeatother  

\section{Introduction}

Relative Property~(T) is an analogue of Kazhdan's Property~(T) for pairs $(G,H)$, where $H$~is a closed subgroup of the locally compact group~$G$. More precisely, $(G,H)$ has \emph{\relT} if every unitary representation of~$G$ with almost-invariant vectors has $H$-invariant vectors. (See \cref{RelTDefn}. Additional information can be found in \cite[pp.~41--43]{BHV}, \cite{Cornulier}, and \cite{Jaudon}.) This concept has proved useful for many purposes, including the study of finitely-additive measures on Euclidean spaces \cite{Margulis-FinitelyAdditive}, the construction of $\mathrm{II}_1$~factors with trivial fundamental group \cite{Popa}, the construction of new examples of groups with Kazhdan's Property~(T) that satisfy the Baum-Connes Conjecture \cite{Valette-GroupPairs}, and proving that particular groups have Kazhdan's Property~(T). 
In particular, the usual proof that $\SL(3,\RR)$ has Kazhdan's Property~(T) is based on the fact that the pair $\bigl( \SL(2,\RR) \ltimes \RR^2, \RR^2 \bigr)$ has \relT\ \cite[pp.~47--50]{BHV}.

The very basic case where the subgroup~$H$ is abelian and normal has been a focus of attention (see, for example, \cite{CornulierTessera,CornulierValette,Fernos,Iozzi,Valette-GroupPairs} and \cite[Lem.~3.1]{Wang}). We generalize the results that were obtained in this situation by allowing $H$ to be nilpotent, rather than abelian.
Indeed, the following \lcnamecref{Nilpotent/N1} provides a nilpotent analogue of any result that establishes \relT\ for abelian, normal subgroups.

\begin{notation}
For any topological group~$N$, we let $N^{(1)} = \cl{[N,N]}$ be the closure of the commutator subgroup of~$N$, and let $N^{ab} = N/N^{(1)}$ be the abelianization of~$N$.
\end{notation}

\begin{thm} \label{Nilpotent/N1}
Let $N$ be a closed, nilpotent, normal subgroup of a locally compact group~$G$.
 Then $(G,N)$ has \relT\ if and only if $(G/N^{(1)}, N^{ab})$ has \relT.
\end{thm}

As an example, consider a semidirect product $H \ltimes A$, where $A$~is abelian. Y.~Cornulier and R.~Tessera \cite{CornulierTessera} have characterized precisely when the pair $(H \ltimes A, A)$ has \relT, so the \lcnamecref{Nilpotent/N1} yields a characterization for pairs $(H \ltimes N, N)$, where $N$~is nilpotent. The following \lcnamecref{SemiNilpotent} is a special case that is in a particularly usable form, and is based on work of Y.~Cornulier and A.~Valette \cite{CornulierValette}.

\begin{notation} \label{IntOnAbel}
Assume the locally compact group~$H$ acts on a $1$-connected, nilpotent Lie group~$N$, and $L$~is a closed, connected, $H$-invariant subgroup of~$N$, such that $[N,N] \subseteq L$. Then $N/L \cong \RR^n$ for some~$n$, so the action of~$H$ induces a homomorphism $\Int_{N/L} \colon H \to \GL(n,\RR)$. 
 We use $\Int_{N/L}(H)^\bullet$ to denote the \emph{closure} of the image of this homomorphism.
\end{notation}

\begin{cor} \label{SemiNilpotent}
Assume the locally compact group~$H$ acts on a $1$-connected, nilpotent Lie group~$N$. The pair $(H \ltimes N, N)$ has \relT\ if and only if, for every closed, connected, $H$-invariant, proper subgroup~$L$ of~$N$ that contains~$N^{(1)}$, the group $\Int_{N/L}(H)^\bullet$ is \underline{not} amenable. 
\end{cor}

A special case of \cref{SemiNilpotent}, in which $H$ is a Lie group and other assumptions are also made, was proved in \cite[Prop.~4.1.4, p.~44]{HaagerupProperty}.

The above results are consequences of the following \lcnamecref{relT(GH1meas)}, which is of independent interest.

\begin{thm} \label{relT(GH1meas)}
Let $H$ be a closed subgroup of a locally compact group~$G$, and let $A$ be a closed, abelian subgroup of~$H$. Assume that $A$ is normal in~$G$, and that every $H$-invariant finite measure on the unitary dual~$\dual A$ is supported on the set of fixed points of~$H$. If $(G/A, H/A)$ has \relT, then $(G, H^{(1)})$ has \relT. 
\end{thm}

The (easy) proof of \Cref{Nilpotent/N1} does not require the full generality of \cref{relT(GH1meas)}, but only the following special case in which $H$~acts trivially on~$A$.

\begin{cor} \label{relT(GH1)}
Let $H$ be a closed subgroup of a locally compact group~$G$, and let $A$ be a closed subgroup of the center of~$H$, such that $A$ is normal in~$G$. If $(G/A, H/A)$ has \relT, then $(G, H^{(1)})$ has \relT. 
\end{cor}

\begin{rem} \label{relT(GH1)Serre}
The special case of \cref{relT(GH1)} in which $G = H$ is a well-known result of J.--P.~Serre that appears in \cite[Thm.~1.7.11, p.~66]{BHV}. More generally, the special case where $A$~is central in all of~$G$, not merely in~$H$, is a generalization of \cite[Prop.~3.1.3]{Cornulier}.
\end{rem}

Our methods also apply to \relT\ for triples, rather than pairs.

\begin{defn}[{}{\cite[Rem.~0.2.2, p.~3]{Jaudon}}] \label{RelTTriple}
Let $H$ and~$M$ be closed subgroups of a locally compact group~$G$. We say that the triple $(G,H,M)$ has \emph{\relT}, if for any unitary representation $\pi$ of~$G$, such that the restriction $\pi|_H$ has almost-invariant vectors, then there exist nonzero $\pi(M)$-invariant vectors.
\end{defn}

For example, we prove the following result, which was conjectured by C.~R.~E.~Raja \cite[Conjecture~1 of \S7]{Raja} in the special case where $N$ is required to be a connected Lie group (in addition to being nilpotent).

\begin{cor} \label{StrongT}
Suppose that $H$ and~$N$ are locally compact groups, such that $N$ is nilpotent and assume that $H$~acts on~$N$ by automorphisms. Then the triple $(H \ltimes N, H, N)$ has \relT\ if and only if the triple $( H \ltimes N^{ab}, H, N^{ab})$ has \relT.
\end{cor}

A modified version of \cref{relT(GH1meas)} also yields a classification of Kazhdan sets in some groups.

\begin{defn}
A subset~$Q$ of a locally compact group~$G$ is a \emph{Kazhdan set} for~$G$ if there exists $\epsilon > 0$, such that every unitary representation of~$G$ with a nonzero $(Q,\epsilon)$-invariant vector has a nonzero invariant vector. (See \fullcref{RelTDefn}{Qe} for the definition of a $(Q,\epsilon)$-invariant vector.)
\end{defn}


C.~Badea and S.~Grivaux \cite[Thm.~8.4]{BadeaGrivaux} obtained a Fourier-analytic characterization of Kazhdan sets in abelian groups (that are locally compact).
The following \lcnamecref{KazhdanSetForNilp} extends this to two other classes of groups.
(The special case where $G$~is a Heisenberg group was proved by C.~Badea and S.~Grivaux \cite[Thm.~8.12]{BadeaGrivaux}.)

\begin{defn}
A connected Lie group~$G$ is \emph{real split} if every eigenvalue of $\mathop{\mathrm{Ad}} g$ is real, for every $g \in G$. For example, every connected, nilpotent Lie group is real split.
\end{defn}

\begin{cor} \label{KazhdanSetForNilp}
Let $G$ be a locally compact group that either is nilpotent or is a connected, real split, solvable Lie group.
Then a subset~$Q$ of~$G$ is a Kazhdan set for~$G$ if and only if the image of~$Q$ in $G^{ab}$ is a Kazhdan set for~$G^{ab}$.
\end{cor}

\begin{rem} \label{GenToSubsets}
Y.~Cornulier \cite[p.~302]{Cornulier} has generalized the notion of \relT\ to pairs $(G,H)$ in which $H$ is a sub\emph{set} of~$G$, rather than a sub\emph{group}. \Cref{relT(GH1)} extends to this setting in the obvious way \csee{relT(GH1)Subset}, but the hypotheses of the corresponding generalization of \cref{relT(GH1meas)} are not as clean \ccf{TripleHasTSubset}. 
\end{rem}

Other consequences of \cref{relT(GH1meas)} can be found in \Cref{IntroProofSect,GeneralizeSemiNilpotentSect,OtherSect}.

Here is an outline of the paper. \Cref{relTriples} establishes some notation and recalls (or proves) several basic facts about \relT, introducing the notion of \relT\ with approximation. \Cref{tensor} defines a tensor product that is fibered over the eigenspaces of an abelian normal subgroup, and discusses the associated invariant or almost-invariant vectors. 
\Cref{MainPfSect} uses the results of \cref{relTriples,tensor} to give a short proof of a generalization of \cref{relT(GH1meas)} that applies to triples, rather than pairs. (The \lcnamecref{MainPfSect} also proves a slightly different result that also implies \cref{Nilpotent/N1}.)
\Cref{IntroProofSect} uses \cref{relT(GH1meas)} (and its generalizations) to prove the other results stated in the above introduction (plus some related results).
\Cref{LargestRelTSect} shows that if $N$ is compactly generated, and nilpotent, then it has a unique largest subgroup~$L^\dagger$, such that $(G,L^\dagger)$ has \relT.
\Cref{GeneralizeSemiNilpotentSect} proves a generalization of \cref{SemiNilpotent} that does not require the subgroup~$N$ to be a Lie group.
\Cref{SubsetSect} presents results on \relT\ for triples $(G,H,M)$ in which the subset~$M$ is not required to be a subgroup. 
Finally, \cref{OtherSect} records a few other observations about \relT.

\begin{ack}
I.~Chatterji is partially supported by the Institut Universitaire de France (IUF) and ANR Gamme.
R.~Shah would like to thank the National Board for Higher Mathematics (NBHM), DAE, Government of India for a research grant.
We thank the Mathematical Sciences Research Institute (Berkeley, California) for its hospitality, which facilitated this collaboration. We also thank Bachir Bekka, Marc Burger, S.~G.~Dani, Talia Fern\'os, and Alain Valette for discussions on an early attempt of the result. 
The results in \cref{SubsetSect} were prompted by an anonymous referee's suggestion to investigate whether our techniques apply to \relT\ for subsets, not just subgroups.
\end{ack}

\section{Relative Property (T) for pairs and triples}\label{relTriples}

\begin{assump} \label{SecondCountable}
Hilbert spaces and locally compact groups are assumed to be second countable. (So all locally compact groups in this paper are $\sigma$-compact.)
\end{assump}

\begin{defn}[{}{\cite[Defns.~1.1.1 and 1.4.3, pp.~28 and 41]{BHV}}] \label{RelTDefn}
Let $\pi$ be a unitary representation of a locally compact group~$G$ on a Hilbert space~$\hilbert$, and let $H$ be a closed subgroup of~$G$.
	\begin{enumerate}
	\item \label{RelTDefn-Qe}
	For a subset~$Q$ of~$G$ and $\epsilon > 0$, a vector $\xi \in \hilbert$ is \emph{$(Q,\epsilon)$-invariant} if $\| \pi(g) \xi - \xi \| \le \epsilon \|\xi\|$ for all $g \in Q$.
	\item $\pi$ has \emph{almost-invariant vectors} if $\pi$~has nonzero $(Q,\epsilon)$-invariant vectors, for every compact $Q \subseteq G$ and $\epsilon > 0$.
	\item The pair $(G,H)$ has \emph{\relT} if every unitary representation of~$G$ that has almost-invariant vectors, also has nonzero $H$-invariant vectors.
	\end{enumerate}
\end{defn}

If the pair $(G,H)$ has \relT, then $(Q,\epsilon)$-invariant vectors can be approximated by $H$-invariant vectors:

\begin{thm}[Jolissaint {\cite[Thm.~1.2 (a2 $\Rightarrow$ b2)]{Jolissaint}}] \label{NearInvariantForPair}
Assume $H$ is a closed subgroup of a locally compact group~$G$, such that $(G,H)$ has \relT. Then, for every $\delta > 0$, there exist a compact subset~$Q$ of~$G$, and $\epsilon > 0$, such that if $\pi$~is any unitary representation of~$G$ on a Hilbert space~$\hilbert$, and $\xi$~is a nonzero $(Q,\epsilon)$-invariant vector in~$\hilbert$, then $\|\xi - \eta\| < \delta \|\xi\|$, for some $H$-invariant vector $\eta \in \hilbert$.
\end{thm}

This result does not extend to triples with \relT, because the following is an example in which the triple $(G,H,M)$ has \relT, but there are almost-invariant vectors for~$H$ that cannot be approximated by $M$-invariant vectors. 

\begin{eg} \label{TripleNotApprox}
Let $G = \mathrm{O}(n) \ltimes \RR^n$, and let $H$ and~$M$ be the stabilizers in~$G$ of two different points~$x$ and~$y$ in~$\RR^n$ (so $H$ and~$M$ are two different conjugates of~$\mathrm{O}(n)$). Then it is not difficult to see that $(G,H,M)$ has \relT. (Namely, note that $H$ has Property~(T), because it is compact, and that every representation of~$G$ with an $H$-invariant vector must also have an $M$-invariant vector, because $M$ is conjugate to~$H$.)

Let $\pi$ be the natural representation of~$G$ on $L^2(\RR^n)$. There is a nonzero $H$-invariant function~$\xi$ in $L^2(\RR^n)$ whose support is contained in a small disk centered at~$x$ (small enough that the disk does not contain~$y$). Then $\xi$~is $(Q,\epsilon)$-invariant for every $Q \subseteq H$ and $\epsilon > 0$, but $\xi$ is not well approximated by any $M$-invariant function.
\end{eg}

This observation motivates the following definition, which identifies the cases where the approximation is always possible:

\begin{defn} \label{RelTApproxTriple}
Let $H$ and~$M$ be closed subgroups of a locally compact group~$G$. We say that the triple $(G,H,M)$ has \emph{\relT\ with approximation} if, for every $\delta > 0$, there exist a compact subset~$Q$ of~$H$ and $\epsilon > 0$, such that if $\xi$~is any $(Q,\epsilon)$-invariant vector of any unitary representation of~$G$, then there is an $M$-invariant vector~$\eta$, such that $\| \eta - \xi \| \le \delta \|\xi\|$.
\end{defn}

It is obvious that \relT\ with approximation implies \relT.  The converse is not true, as Example \ref{TripleNotApprox} gives a triple that has \relT\ but not \relT\ with approximation. However, \cref{NearInvariantForPair} tells us that the two properties are equivalent when $G = H$. They are also equivalent when the third group in the triple is normal:

\begin{lem} \label{NearInvariantForNormal}
Assume $H$ and $M$ are closed subgroups of a locally compact group~$G$, such that $(G,H,M)$ has \relT. If $M$ is normal in~$G$, then $(G,H,M)$ has \relT\ with approximation.
\end{lem}

\begin{proof}
This is a standard argument (cf.\ \cite[Prop.~1.1.9, p.~31]{BHV}). Let $\delta > 0$ be arbitrary.
Since $(G,H,M)$ has \relT, there exist a compact subset~$Q$ of~$H$ and $\epsilon' > 0$, such that every unitary representation of~$G$ with nonzero $(Q,\epsilon')$-invariant vectors has nonzero $M$-invariant vectors. 

Let $\epsilon = \delta \epsilon'/2$, and suppose that $\xi$ is a $(Q,\epsilon)$-invariant unit vector for a unitary representation~$\pi$ of~$G$ on a Hilbert space~$\hilbert$. We need to find an $M$-invariant vector~$\eta$ that is $\delta$-close to $\xi$.

Let $P \colon \hilbert \to (\hilbert^M)^\perp$ be the projection onto the orthogonal complement of the space of $M$-invariant vectors. We may assume $P(\xi) \neq 0$ (otherwise $\xi$ is invariant and we take $\eta = \xi$). Since $M$ is normal in~$G$, we know that $\hilbert^M$ is $G$-invariant, so $\pi$~restricts to a representation of~$G$ on $(\hilbert^M)^\perp$. For all $q \in Q$, we have
	$$ \| \pi(q) \, P(\xi) - P(\xi) \|
	= \|P \bigl( \pi(q) \xi - \xi \bigr) \|
	\le \| \pi(q) \xi - \xi \|
	\le \epsilon .$$
However, $P(\xi)$ cannot be $(Q,\epsilon')$-invariant, because $(\hilbert^M)^\perp$ has no nonzero $M$-invariant vectors. Therefore $\epsilon > \epsilon' \| P(\xi) \|$, which means 
	\begin{align*}
	\|P(\xi)\| &< \epsilon/\epsilon' = \delta/2 < \delta \|\xi\|
	. 
	\end{align*}
Hence $\eta = \xi - P(\xi) \neq 0$ is $M$-invariant and $\|\eta-\xi\|\leq\delta\|\xi\|$ as desired.
\end{proof}

It is immediate from the definitions that the pair $(G,H)$ has \relT\ if and only if the triple $(G,G,H)$ has \relT.
Now, suppose $M \subseteq H \subseteq G$. It is obvious that if the pair $(H,M)$ has \relT, then the triple $(G,H,M)$ has \relT. However, the converse is not true, even if $M$ is contained in~$H$ and is normal in~$G$:

\begin{eg} \label{TripleNotPair}
Fix $n \ge 4$, and embed $\SL(3,\RR)$ in $\SL(n,\RR)$, in such a way that $\SL(3,\RR)$ fixes a nonzero vector $v \in \RR^n$. Then 
	\begin{enumerate}
	\item \label{TripleNotPair-triple}
	the triple $\bigl( \SL(n,\RR) \ltimes \RR^n, \SL(3,\RR) \ltimes \RR^n, \RR^n \bigr)$ has \relT, 
	but 
	\item \label{TripleNotPair-pair}
	the pair $\bigl( \SL(3,\RR) \ltimes \RR^n, \RR^n \bigr)$ does not have \relT.
	\end{enumerate}
\end{eg}

\begin{proof}
\pref{TripleNotPair-triple} Let $\pi$ be a unitary representation of $\SL(n,\RR) \ltimes \RR^n$, such that the restriction of~$\pi$ to $\SL(3,\RR) \ltimes \RR^n$ has nonzero almost-invariant vectors. Since $\SL(3,\RR)$ has Property~(T) \cite[Thm,~1.4.15, p.~49]{BHV}, we know that $\pi$ has nonzero $\SL(3,\RR)$-invariant vectors. The Moore Ergodicity Theorem (or Mautner phenomenon) \cite[Cor.~11.2.8, p.~216]{Morris} tells us that every $\SL(3,\RR)$-invariant vector is $\SL(n,\RR)$-invariant. Since the triple $\bigl( \SL(n,\RR) \ltimes \RR^n, \SL(n,\RR), \RR^n \bigr)$ has \relT\ (see, for example, \cite[Thm.~1.1]{Raja}), these vectors are $\RR^n$-invariant.

\pref{TripleNotPair-pair} Since $\SL(3,\RR)$ fixes~$v$ (and $\SL(3,\RR)$ is simple, so its representation on~$\RR^n$ is completely reducible), we see that the abelianization of $\SL(3,\RR) \ltimes \RR^n$ is noncompact. So nontrivial $1$-dimensional representations of $\SL(3,\RR) \ltimes \RR^n$ approximate the trivial representation, and are trivial on $\SL(3,\RR)$, but have no $\RR^n$-invariant vectors.
\end{proof}

\section{Invariant vectors and tensor products} \label{tensor}
As was mentioned in the introduction, \cref{relT(GH1meas)} is a generalization of a theorem of Serre. The proof of Serre's result in \cite[Thm.~1.7.11, p.~66]{BHV} is based on the fact that if $A$~is central in~$G$, and $\pi$ is irreducible, then $\pi(A)$~consists of scalar matrices, so $A$~is in the kernel of $\pi \otimes \conj\pi$ (see \cref{ConjNotation} for the definition of the conjugate representation~$\conj\pi$).
To generalize this proof, we construct a different representation, denoted $\pi \otimes_{\dual A} \conj\pi$, that is trivial on~$A$, even if $\pi(A)$ does not consist of scalars \csee{FiberDefn}. 
In geometric terms, $\pi$~can be realized as an action on the $L^2$-sections of a vector bundle over the unitary dual of~$A$, and the representation $\pi \otimes_{\dual A} \conj\pi$ is constructed by tensoring this vector bundle with its conjugate. However, the official definition of $\pi \otimes_{\dual A} \conj\pi$ in \cref{FiberSect} uses the terminology of real analysis and representation theory, instead of the language of vector bundles.

For the proof of \cref{relT(GH1meas)}, it is important to know that almost-invariant vectors for~$\pi$ yield almost-invariant vectors for $\pi \otimes_{\dual A} \conj\pi$. That is the point of \cref{TensorAlmInv} below. Conversely, \cref{MMinv} will be used to obtain invariant vectors for~$\pi$ from invariant vectors for $\pi \otimes_{\dual A} \conj\pi$.

\begin{notation} \label{ConjNotation}
We use:
	\begin{itemize}
	\item $\conj{x}$ for the complex conjugate of the number~$x$,
	\item $\conj{\hilbert}$ for the conjugate of the Hilbert space~$\hilbert$ \cite[p.~293]{BHV},
	\item $\conj{\xi}$ for the element of~$\conj{\hilbert}$ corresponding to the element~$\xi$ of~$\hilbert$ (so $\conj{x \xi} = \conj{x} \, \conj{\xi}$ for $x \in \CC$),
	and
	\item $\conj{\pi}$ for the unitary representation on~$\conj{\hilbert}$ that is obtained from the unitary representation~$\pi$ on~$\hilbert$ \cite[Defn.~A.1.10, p.~294]{BHV}.
	\end{itemize}
\end{notation}

We begin by recalling some basic facts of functional analysis.

\begin{lem} \ \label{FuncAnal}
\noprelistbreak
	\begin{enumerate}
	\item \cite[\S3.4, pp.~42--49]{Weidmann}
	If $\hilbert_1$ and $\hilbert_2$ are Hilbert spaces, then there is a Hilbert space $\hilbert_1 \otimes \hilbert_2$, such that
                $$ \| v_1 \otimes v_2 \| = \|v_1\| \, \|v_2\|$$
         for all $v_1 \in \hilbert_1$ and $v_2 \in \hilbert_2$.
 
        \item \textup{(}cf.\ \cite[p.~267 and Thm.~3.12(b), p.~49]{Weidmann}\textup{)}
        If $U_1$ and $U_2$ are unitary operators on $\hilbert_1$ and $\hilbert_2$, respectively, then there is a unitary operator $U_1 \otimes U_2$ on $\hilbert_1 \otimes \hilbert_2$, such that
                $$(U_1 \otimes U_2) (v_1 \otimes v_2) = U_1 v_1 \otimes U_2 v_2 .$$
 
        \item \label{FuncAnal-TensorCont} 
        The natural map $\unitary(\hilbert_1) \times \unitary(\hilbert_2) \to \unitary(\hilbert_1 \otimes \hilbert_2)$ is continuous%
        \refnote{FuncAnal-TensorContPf}
        when the unitary groups are given the strong operator topology.
	\end{enumerate}
\end{lem}

\subsection{A fibered tensor product} \label{FiberSect}

\begin{notation}
Assume 
	\begin{itemize}
	\item $\pi$ is a unitary representation of a locally compact group~$G$,
	and
	\item $A$ is an abelian, normal subgroup of~$G$.
	\end{itemize}
\end{notation}

Applying the representation theory of abelian groups \cite[Thm.~D.3.1(i), p.~375]{BHV} to the restriction $\pi|_A$ provides a unique projection-valued measure~$\PP$ on the unitary dual~$\dual A$, such that, for $a \in A$, we have
	$$\pi(a) = \int_{\dual A} \lambda(a) \, d \mkern0.4\thinmuskip \PP(\lambda) .$$
The uniqueness implies that
        $$ \text{$\PP_{gE} = \pi(g) \, \PP_E \, \pi(g)^{-1}$ for $g \in G$ and $E \subseteq \dual A$} ,$$
so this is a system of imprimitivity for~$\pi$ (as defined in \cite[top of page 203]{Varadarajan}).  If this system of imprimitivity is homogeneous (as defined in \cite[p.~218]{Varadarajan}), then \cite[Theorem 6.11, pp. 220--221]{Varadarajan} tells us there is a measure~$\mu$ on~$\dual A$, a Hilbert space~$\hilbert$, and a Borel cocycle $\alpha \colon G \times \dual A \to \unitary(\hilbert)$, such that (up to isomorphism) $\pi$ is the representation on $L^2(\dual A, \mu ; \hilbert)$ given by 
	$$ \bigl( \pi(g) f \bigr)(\lambda) = \sqrt{D(g,\lambda)} \, \alpha(g,\lambda) \, f(g^{-1}\lambda) $$
for $g \in G$, $f \in L^2 ( \dual A, \mu ; \hilbert)$, and $\lambda \in \dual A$, and where $D(g,\lambda)$ is the Radon-Nikodym derivative of the action of~$g$ on~$\dual A$. 

\begin{defn} \label{FiberDefn}
With the above notation (and assuming that $\PP$ is homogeneous), we define $\pi' = \pi \otimes_{\dual A} \conj{\pi}$ to be the unitary representation of~$G$ on $L^2 ( \dual A, \mu ; \hilbert \otimes \conj{\hilbert})$ that is defined by replacing $\alpha$ with $\alpha \otimes \conj{\alpha}$ in the formula for $\pi(g)$:
	$$ \bigl( \pi'(g) f \bigr)(\lambda) = \sqrt{D(g,\lambda)} \, \bigl( \alpha(g,\lambda) \otimes \conj{\alpha(g,\lambda)} \bigr) \, f(g^{-1}\lambda) . $$
(\fullCref{FuncAnal}{TensorCont} implies that the cocycle $\alpha \otimes \conj{\alpha}$ is Borel measurable.)
\end{defn}

\begin{rem}\label{Akerpip} 
Notice that $A$ is in the kernel of $\pi \otimes_{\dual A} \conj{\pi}$,%
\refnote{AinKernel}
which means that $\pi \otimes_{\dual A} \conj{\pi}$ is a representation of $G/A$.
\end{rem}

An important feature of the fibered tensor product is that it preserves almost-invariant vectors, and more precisely we have the following.

\begin{prop} \label{TensorAlmInv}
If $f \in L^2( \dual A, \mu ; \hilbert)$ is a $(Q,\epsilon/3)$-invariant unit vector for~$\pi$, then  the function $f'(\lambda) = \frac{1}{\|f(\lambda)\|} \, f(\lambda) \otimes \conj{f(\lambda)}$ is a $(Q,\epsilon)$-invariant unit vector for $\pi \otimes_{\dual A} \conj{\pi}$.
\textup(\!We use the convention that $f'(\lambda) = 0$ if $f(\lambda) = 0$.\textup)
\end{prop}

Before giving the proof we need the following simple lemma.

\begin{lem} \label{3Triangle}
Suppose
	\begin{itemize}
	\item $\hilbert_1$ and $\hilbert_2$ are Hilbert spaces,
	\item $v_i,w_i \in \hilbert_i$ for $i = 1,2$,
	\item $\|w_1\| = \|w_2\|$,
	\item for $z \in \{v_1,v_2,w_1,w_2\}$, $\widehat z$ is a unit vector such that $z = \|z\| \, \widehat z$,
	and
	\item $D \ge 0$.
	\end{itemize}
Then
	$$ \| D w_1 \otimes \widehat w_2 - \widehat v_1 \otimes v_2 \| \le 2\| D w_1 - v_1 \| + \| Dw_2 - v_2 \| .$$
\end{lem}

\begin{proof}
We have
	\begin{align*}
	 \| D & w_1 \otimes \widehat w_2 - \widehat v_1 \otimes  v_2 \|
	 \\& \le \| D w_1 {\otimes} \widehat w_2 - v_1 {\otimes} \widehat w_2 \|
	 	+ \| v_1 {\otimes} \widehat w_2 - \widehat v_1 {\otimes} (Dw_2) \|
		+ \| \widehat v_1 {\otimes} (Dw_2) - \widehat v_1 {\otimes} v_2 \|
	 \\& = \| (D w_1 - v_1) \otimes \widehat w_2 \|
	 	+ \Bigl\| \bigl( \| v_1 \| - D \| w_2 \| \bigr) \widehat v_1 \otimes \widehat w_2 \Bigr\|
		+ \| \widehat v_1 \otimes (Dw_2 - v_2) \|
	 \\& = \| D w_1 - v_1 \|
	 	+ \Bigl| \| v_1 \| - D \| w_2 \| \Bigr|
		+ \| Dw_2 - v_2 \|
	.\end{align*}
Since $\|w_1\| = \|w_2\|$, the conclusion now follows from the fact that $\bigl| \|v\| - \|w\| \bigr| \le \| v - w \|$ for all vectors $v$ and~$w$ in any Hilbert space.
\end{proof}

\begin{proof}[\bf Proof of Proposition \ref{TensorAlmInv}]
First, note that for $\lambda \in \dual A$, we have
	$$ \| f'(\lambda) \|
	= \left\| \frac{1}{\|f(\lambda)\|} \, f(\lambda) \otimes \conj{f(\lambda)} \right\|
	= \frac{1}{\|f(\lambda)\|} \cdot \| f(\lambda) \| \cdot \| \conj{f(\lambda)} \|
	= \| \conj{f(\lambda)} \|
	= \| f(\lambda) \|
	. $$
Therefore $\|f'\|_2 = \|f\|_2$, so $f'$~is a unit vector for the representation $\pi' = \pi \otimes_{\dual A} \conj{\pi}$.

For $g \in Q$ and $\lambda \in \dual A$, let
	$$ \text{$v_1 = f(\lambda)$, \quad $w_1 = \alpha(g,\lambda) \, f(g^{-1}\lambda)$,  \quad $v_2 = \conj{v_1}$,} $$
	$$  \text{$w_2 = \conj{w_1}$,
	\quad and \quad 
	$D = \sqrt{D(g,\lambda)}$} .$$
Then
	$$ f'(\lambda) = \frac{1}{\| f(\lambda) \|} \,  f(\lambda) \otimes \conj{f(\lambda)} = \widehat v_1 \otimes v_2 .$$
and
	\begin{align*}
	\bigl( \pi'(g)  f' \bigr)  (\lambda)
	&= \sqrt{D(g,\lambda)} \, \alpha'(g,\lambda) \, f'(g^{-1}\lambda)
	\\&= \frac{\sqrt{D(g,\lambda)}}{\|f(g^{-1} \lambda)\|} \, \alpha(g,\lambda) \, f(g^{-1}\lambda) \otimes \conj{\alpha(g,\lambda)} \, \conj{f(g^{-1}\lambda)}
	\\& = D w_1 \otimes \widehat w_2 
	 \end{align*}
Therefore, \cref{3Triangle} tells us that
	\begin{align*}
	 \left\| \bigl( \pi'(g)  f' \bigr)  (\lambda) - f'(\lambda) \right\|
	&\le 2\| D w_1 - v_1 \| + \| Dw_2 - v_2 \|
	\\&= 2\| D w_1 - v_1 \| + \| D\conj{w_1} - \conj{v_1} \|
	\\&= 3\| D w_1 - v_1 \|
	\\&= 3 \, \| \bigl( \pi(g) f\bigr)(\lambda) - f(\lambda) \| 
	, \end{align*}
since
	$$\bigl( \pi(g) f\bigr)(\lambda) = \sqrt{D(g,\lambda)} \, \alpha(g,\lambda) \, f(g^{-1}\lambda) =  D w_1 .$$
So
	\begin{align*}
	 \| \pi'(g)  f' - f' \|_2 \le 3 \, \| \pi(g)  f - f \|_2 < 3 \cdot \frac{\epsilon}{3} = \epsilon 
	. & \qedhere \end{align*}
\end{proof}

\subsection{Obtaining invariant vectors from a tensor product}
The following result is based on ideas of Jolissaint \cite[Thm.~1.2]{Jolissaint} and Nicoara-Popa-Sasyk \cite[proof of Lem.~1]{NPS}.

\begin{prop} \label{MMinv}
Let $\rho$ be a unitary representation of a locally compact group~$M$ on a Hilbert space~$\hilbert$, and suppose $\xi \in \hilbert$. If $\eta'$ is any $(\rho \otimes \conj{\rho})$-invariant vector in $\hilbert \otimes \conj\hilbert$, then there is a $\rho( M^{(1)} )$-invariant vector $\eta \in \hilbert$, such that
	$$ \| \eta - \xi \| \, \| \xi\|  \le 7 \, \| \eta' - \xi \otimes \conj{\xi} \| .$$
Moreover, $\eta$~can be chosen so that the subspace $\CC \eta$ is $\rho(M)$-invariant.
\end{prop}

\begin{proof}
 Assume, without loss of generality, that $\|\xi\| = 1$, and, for convenience, let $\delta = \| \eta' - \xi \otimes \conj{\xi} \|$. We may assume that $\delta < 1/7$. (Otherwise, the desired inequality is satisfied with $\eta = 0$.)

Let $\hilbert' = \hilbert \otimes \conj{\hilbert}$, and note that $\hilbert'$ can be identified with the space of Hilbert-Schmidt operators on~$\hilbert$, which are compact operators with finite trace (see \cite[the discussion on page 294]{BHV}). In this identification, the vector $\xi' = \xi \otimes \conj{\xi}$ corresponds to the rank-one orthogonal projection $P_\xi$ on the line~$\CC\xi$, defined by
	$$ P_\xi(\eta) = \langle\eta,\xi\rangle\xi ,$$
where the inner product is from~$\hilbert$. In particular, $P_\xi$ is a self-adjoint operator with trace~1 and its spectrum is in $\{0,1\}$. Therefore, all the Hilbert-Schmidt operators corresponding to elements in the closed convex hull of $(\rho \otimes \conj{\rho})(M)\xi'$ are self-adjoint operators with trace~1 and their spectrum is contained in $[0, 1]$.

Let $T$ be the Hilbert-Schmidt operator corresponding to~$\eta'$.
We may assume that $\eta'$ is the projection of~$\xi'$ onto the space of $(\rho \otimes \conj{\rho})$-invariant vectors, so $\eta'$ is in the closed convex hull of $(\rho \otimes \conj{\rho})(M)\xi'$. Then $T$ is a self-adjoint Hilbert-Schmidt operator with trace~1 and whose spectrum is contained in $[0, 1]$.

As $T$ is invariant under $\rho(M)$, so are all of its eigenspaces, and we claim that $T$ has a one-dimensional eigenspace. Let $\{c_i\}$ be the eigenvalues of~$T$, and assume $c_1 > c_i$ for all $i \neq 1$. Then $\sum c_i = 1$, and the Hilbert-Schmidt norm~$\|T\|_{HS}$ of~$T$ satisfies
	$$ \text{$\|T\|_{HS} = \left( \sum c_i^2 \right)^{1/2} \le 1$ \ and \ $\|T\|_{HS} \ge \|T\| = c_1$} . $$
From the definition of~$T$, we have
	$$ \| T - P_\xi \|_{HS} = \| \eta' - \xi' \| = \delta .$$
Also
	$$ \| T^2 - P_\xi\|_{HS} = \| T^2 - P_\xi^2\|_{HS} \le \|T+P_\xi\|_{HS} \, \|T-P_\xi\|_{HS} \le 2\delta .$$
Then
	$$ \| T-T^2\|_{HS} \le \|T-P_\xi \|_{HS} + \| P_\xi-T^2\|_{HS} \le \delta + 2\delta =3\delta .$$
As $c_1 > c_i$ for all $i \neq 1$, we have 
	\begin{align*}
	 3\delta 
	 &\ge \|T\|_{HS} - \|T^2\|_{HS} = \left( \sum c_i^2 \right)^{1/2} - \left( \sum c_i^4 \right)^{1/2}  
	 \\&\ge (1-c_1) \left( \sum c_i^2 \right)^{1/2}= (1-c_1) \|T\|_{HS}\ge (1-c_1)(1-\delta)
	 . \end{align*}
Therefore 
	 $$ c_1 \ge 1- \frac{3\delta}{1-\delta} = \frac{1-4\delta}{1-\delta} > \frac{1}{2}, $$
since $\delta < 1/7$. Since the trace of~$T$ is~1, we conclude that the eigenspace corresponding
to the eigenvalue $c_1$ has dimension~1, as claimed.

Note that the $c_1$-eigenspace of~$T$ is not orthogonal to~$\xi$: otherwise, if $\eta_0$ is a unit vector in the eigenspace, then
	$$ \frac{1}{2} < c_1 = \| c_1 \eta_0 \| = \| T\eta_0 - P_\xi(\eta_0) \| \le \| T - P_\xi \| \le \delta < \frac{1}{7} .$$
Therefore, we may let $\eta$ be the (unique) vector in the $c_1$-eigenspace of~$T$, such that $P_\xi(\eta) = c_1\xi$. Then
	$$ \| \eta - \xi \|
	=  \frac{1}{c_1}\| T(\eta)  - P_\xi(\eta) \| 
	\le  \frac{1}{c_1}\| T - P_\xi \| \, \| \eta\|
	<  \frac{\delta}{c_1} \| \eta \|
	< 2 \delta \| \eta \|
	.$$
Since $\|\xi\| = 1$ and $2\delta < 2/7 < 1/3$, this implies 
	$ \| \eta \| < 3/2 $,
so $\| \eta - \xi \| < 3 \delta < 7 \delta$. Also, since $\CC \eta$ is a $1$-dimensional $\rho(M)$-invariant subspace, we know that $\rho(M^{(1)})$ acts trivially on it, so $\eta$ is $\rho(M^{(1)})$-invariant.
\end{proof}

\section{Proof of the main theorem} \label{MainPfSect}

Recall that all locally compact groups are assumed to be second countable \csee{SecondCountable}.

\Cref{relT(GH1meas)} is the special case of the following result in which $G = H$. (Recall that if either $G = H$ or $M$~is normal in~$G$, then \relT\ for the triple $(G,H,M)$ is equivalent to \relT\ with approximation, by \cref{NearInvariantForPair,NearInvariantForNormal}.)

\begin{thm} \label{TripleHasT}
Let $H$ and $M$ be closed subgroups of a locally compact group~$G$, and let $A$ be a closed, abelian subgroup of~$M$ that is normal in~$G$. Assume that every $M$-invariant finite measure on~$\dual A$ is supported on the set of fixed points of~$M$. If $HA$~is closed and\/ $(G/A, HA/A, M/A)$ has \relT\ with approximation, then\/ $(G, H, M^{(1)})$ has \relT\ with approximation. 
\end{thm}

\begin{proof}
Let $\delta > 0$ be arbitrary. Since $(G/A, HA/A, M/A)$ has \relT\ with approximation (and $HA$ is closed), there is a compact subset~$Q$ of~$H$ and $\epsilon > 0$, such that if $\xi'$ is any $(Q,\epsilon)$-invariant vector for a unitary representation of $G/A$, then there is an $M$-invariant vector~$\eta'$, such that $\| \eta' - \xi' \| < \delta/7$. 

Now, suppose $\pi$ is a unitary representation of~$G$, such that $\pi$ has a $(Q,\epsilon/3)$-invariant vector~$f$. (We wish to show that $f$ is well-approximated by an $M$-invariant vector.) By replacing $\pi$ with the direct sum $\pi \oplus \pi \oplus \cdots$ of infinitely many copies of itself, we may assume that all irreducible representations appearing in the direct integral decomposition of~$\pi|_A$ have the same multiplicity (namely,~$\infty$). By definition, this means that $\pi|_A$ is homogeneous, so \cref{FiberSect} provides a measure~$\mu$ on~$\dual A$, a Borel cocycle $\alpha \colon G \times \dual A \to \unitary(\hilbert)$, a corresponding realization of~$\pi$ as a representation on the Hilbert space $L^2(\dual A, \mu;\hilbert)$, and a unitary representation $\pi' = \pi \otimes_{\dual A} \conj{\pi}$ of~$G$. 

Since $\pi$ has been realized as a representation on $L^2( \dual A, \mu ; \hilbert)$, we know that $f \in L^2( \dual A, \mu ; \hilbert)$. Then \cref{TensorAlmInv} provides a $(Q,\epsilon)$-invariant unit vector~$f'$ for~$\pi'$. By the choice of $Q$ and~$\epsilon$ (and Remark \ref{Akerpip}), we know that there is a $\pi'(M)$-invariant vector $f'_M\in L^2 ( \dual A, \mu ; \hilbert \otimes \conj{\hilbert} )$, such that $\| f'_M - f' \|_2 < \delta/7$.

Since $f'_M$ is $M$-invariant, we have
$$\|f'_M(\lambda)\|=\sqrt{D(m,\lambda)} \, \|f'_M(m^{-1}\lambda)\| \hbox{ for }m\in M\hbox{ and a.e. }\lambda\in\dual A,$$
so it is straightforward to check that $\|f'_M\|^2 \, \mu$ is an $M$-invariant measure on~$\dual A$. Furthermore, this measure is finite because $f'_M$ is in $L^2 ( \dual A, \mu ; \hilbert \otimes \conj{\hilbert} )$. By assumption, this implies that (up to modifying $f'_M$ on a set of measure zero) we may choose the support of $f'_M$ to be contained in the set $\dual A^M$ of fixed points of~$M$ in~$\dual A$ . For each fixed $\lambda\in\dual A^M$, the function $\alpha(m,\lambda)$ is a representation $\rho_\lambda$ of $M$ on $\hilbert$ (and $D(m,\lambda)\equiv 1$ on $\dual A^M$), so $M$ acts on the subspace $L^2 ( \dual A^M, \mu|_{\dual A^M} ; \hilbert \otimes \conj{\hilbert} )$ by
$$ \bigl( \pi'(m)f' \bigr)(\lambda)= \bigl( (\rho_\lambda \otimes \conj{\rho_\lambda}(m) \bigr) f'(\lambda).$$
Note that $f'_M(\lambda)$ must be an $M$-invariant vector in $\hilbert \otimes \conj\hilbert$, for a.e.\ $\lambda \in \dual A$. Then, since \cref{TensorAlmInv} tells us that $f'(\lambda) = \frac{1}{\|f(\lambda)\|} \, f(\lambda) \otimes \conj{f(\lambda)}$, \cref{MMinv} provides a $\rho_\lambda(M^{(1)})$-invariant vector $v(\lambda) \in \hilbert$, such that $\| v(\lambda) - f(\lambda) \| \le 7 \, \| f'_M(\lambda) - f'(\lambda) \|$. (Also note that the von~Neumann Selection Theorem \cite[Thm.~3.4.3, p.~77]{Arveson} implies that we may choose $v(\lambda)$ to be a measurable function of~$\lambda$.) Then $v$ is $\pi(M^{(1)})$-invariant, and 
	$\| v - f \|_2 \le 7 \|f'_M - f' \|_2 < \delta$.
\end{proof}

\begin{rem} \label{WhenInvMeas}
Here are two situations that satisfy \cref{TripleHasT}'s assumption that every $M$-invariant finite measure on~$\dual A$ is supported on the set of fixed points of~$M$:
	\begin{enumerate}
	\item \label{WhenInvMeas-center}
	If $A$ is contained in the center of~$M$, then $M$~acts trivially on~$A$ (so it also acts trivially on~$\dual A$), so every point in~$\dual A$ is a fixed point.
	\item \label{WhenInvMeas-split}
	If $M$ is a connected, solvable Lie group that is real split, and $A$~is a closed, $1$-connected, abelian, normal subgroup (so $\dual A \cong A$), then the desired conclusion is a well known result in the spirit of the Borel Density Theorem (cf.\ \cite[Cor.~1.3]{Dani-quasi}).%
	\refnote{BDT}
	\end{enumerate}
\end{rem}

\begin{rem} \label{QuantitativeTripleHasT}
If $G$, $H$, $M$, and~$A$ are as described in the first two sentences of the statement of \cref{TripleHasT}, then
the proof establishes the following quantitative version of the \lcnamecref{TripleHasT}: 
	{\it Suppose $Q \subseteq H$ and $\delta,\epsilon > 0$, such that, for every $(Q,\epsilon)$-invariant vector~$\xi'$ for a unitary representation of $G/A$, there is an $M$-invariant vector~$\eta'$, such that $\| \eta' - \xi' \| < \delta/7$. Then, for every $(Q,\epsilon/3)$-invariant vector~$\xi$ for a unitary representation of~$G$, there is an $M^{(1)}$-invariant vector~$\eta$, such that $\| \eta - \xi \| < \delta$.}
\end{rem}

\begin{rem} \label{KazhdanSetForG/A}
Taking $G = H = M$ in \cref{QuantitativeTripleHasT} establishes that if $Q$ is a subset of a locally compact group~$G$, and $A$ is a closed, abelian, normal subgroup of $G$, such that
	\begin{enumerate}
	\item the image of~$Q$ in $G/A$ is a Kazhdan set for $G/A$, 
	and
	\item every $G$-invariant finite measure on~$\dual A$ is supported on the set of fixed points of~$G$,
	\end{enumerate}
then $Q$ is a Kazhdan set for the pair $(G, G^{(1)})$. (That is, there exists $\epsilon > 0$, such that every unitary representation of~$G$ that has $(Q,\epsilon)$-invariant vectors also has $G^{(1)}$-invariant vectors.) \end{rem}

The following \lcnamecref{TripleNilpHasT} removes the phrase ``with approximation'' from the statement of \cref{TripleHasT}, at the expense of placing restrictions on $M$ and~$A$.

\begin{notation}
We use $Z(M)$ to denote the center of a group~$M$.
\end{notation}

\begin{thm} \label{TripleNilpHasT}
Let $H$ and~$M$ be closed subgroups of a locally compact group~$G$, and let $A$ be a closed subgroup of $Z(M) \cap M^{(1)}$ that is normal in $G$. Assume that $HA$ is closed, that $M$ is nilpotent, and that either $M$~has no closed, proper subgroups of finite index, or $A$ is contained in~$H$ and is compactly generated. If\/ $(G/A, HA/A, M/A)$ has \relT, then\/ $(G, H, M)$ has \relT. 
\end{thm}

\begin{proof}
Let $\pi$ be a unitary representation of~$G$, such that $\pi|_H$ has almost-invariant vectors. 

Assume, for the moment, that the triple $(G,H,A)$ has \relT. Then the space of $A$-invariant vectors is nonzero.
Since $A$ is a normal subgroup, this space is $G$-invariant, and therefore yields a representation~$\pi^A$ of~$G/A$. Also (because $A$ is a normal subgroup), \cref{NearInvariantForNormal} tells us that $(G,H,A)$ has \relT\ with approximation, so the restriction of $\pi^A$ to~$H$ has almost-invariant vectors. Since $(G/A,HA/A,M/A)$ has \relT, we conclude that $\pi^A$ (and hence~$\pi$) has nonzero $M$-invariant vectors. So $(G,H,M)$ has \relT, as desired.

To complete the proof, we show that the triple $(G,H,A)$ does indeed have \relT. 
That is, we show that $\pi$~has nonzero $A$-invariant vectors.
Arguing as in the proof of \cref{TripleHasT}, we see that we may assume that $\pi|_A$ is homogeneous (by replacing $\pi$ with $\pi \oplus \pi \oplus \cdots$), so \cref{FiberSect} provides a measure~$\mu$ on~$\dual A$, a Borel cocycle $\alpha \colon G \times \dual A \to \unitary(\hilbert)$, a corresponding realization of~$\pi$ as a representation on the Hilbert space $L^2(\dual A, \mu; \hilbert)$, and a unitary representation $\pi' = \pi \otimes_{\dual A} \conj{\pi}$ of~$G$. Also, $M$~acts trivially on~$\dual A$, so, for each fixed $\lambda \in \dual A$, the function $\alpha(m,\lambda)$ is a representation~$\rho_\lambda$ of~$M$ on~$\hilbert$, so $M$ acts on $L^2 ( \dual A, \mu ; \hilbert \otimes \conj{\hilbert} )$ by
	$$ \bigl( \pi'(m)f' \bigr)(\lambda) = \bigl( (\rho_\lambda \otimes \conj{\rho_\lambda})(m) \bigr)f'(\lambda) .$$
Also, since $(G/A,HA/A,M/A)$ has \relT, there is a nonzero $\pi'|_M$-invariant vector $f'_M$ in $L^2 ( \dual A, \mu ; \hilbert \otimes \conj{\hilbert} )$.

Therefore, $\rho_\lambda \otimes \conj{\rho_\lambda}$ has a nonzero $M$-invariant vector for all~$\lambda$ in a set~$E$ of positive measure. For each $\lambda \in E$, there must be a finite-dimensional $\rho_\lambda(M)$-invariant subspace~$F_\lambda$ of~$\hilbert$ \cite[Prop.~A.1.2, p.~295]{BHV}.
Let $M_\lambda$ be the closure of $\rho_\lambda(M)|_{F_\lambda}$. This is a closed (hence compact) subgroup of $\mathrm{SU}(n)$, for some $n \in \NN$. Every compact, nilpotent Lie group is virtually abelian \cite[Cor.~11.2.11, p.~447]{HilgertNeeb}, so we see that $M_\lambda$ has a closed, abelian subgroup of finite index. 

\setcounter{case}{0}

\begin{case}
Assume $M$ has no closed, proper subgroups of finite index.
\end{case}
Then the entire group $M_\lambda$ must be abelian. This means that $\rho_\lambda(M^{(1)})|_{F_\lambda}$ is trivial. Since $A\subseteq M^{(1)}$, this implies that $\rho_\lambda(A)$ fixes every element of~$F_\lambda$. So $\lambda(a) = 1$ for all $a \in A$ and all $\lambda \in E$.

This means that $\mu \bigl( \{1\} \bigr) \neq 0$. Therefore, if we fix any nonzero $\xi_0 \in \hilbert$, then the function 
	$$f(\lambda) = \begin{cases}
	\xi_0 & \text{if $\lambda = 1$} , \\
	0 & \text{otherwise}
	 \end{cases} $$
is nonzero in $L^2(\dual A, \mu; \hilbert)$. And it is obviously fixed by~$A$. So $\pi$ has a nonzero $A$-invariant vector, as desired.

\begin{case}
Assume $A \subseteq H$ and $A$~is compactly generated.
\end{case}
Since $M_\lambda$ is virtually abelian, we know there is a finite-index subgroup~$M_\lambda'$ of~$M$, such that $\rho_\lambda(M_\lambda')|_{F_\lambda}$ is abelian. So $\rho_\lambda|_{F_\lambda}$ is trivial on $(M_\lambda')^{(1)}$, which is a finite-index subgroup of~$M^{(1)}$ (since $M$ is nilpotent and $M_\lambda'$~has finite index in~$M$). Since $A \subseteq M^{(1)}$, we conclude that $\rho_\lambda|_{F_\lambda}$ is trivial on a finite-index subgroup~$A_\lambda$ of~$A$.

Therefore, there is some $m \in \NN$, such that $\rho_\lambda$ is trivial on~$A^m$ for all $\lambda$ in a set~$E$ of positive measure (where $A^m = \cl{\{a^m \mid a \in A\}}$). This means $\rho_\lambda(A^m)$ fixes every element of~$F_\lambda$ (for all $\lambda \in E$), so $\pi(A^m)$ has a nonzero fixed vector. This implies that $(G,H,A^m)$ has \relT.

Note that $A^m$ is normal in~$G$ (because it is characteristic in the normal subgroup~$A$). 
Also, the quotient $A/A^m$ is compact (because $A$ is compactly generated and abelian). Therefore $(G/A^m, H/A^m, A/A^m)$ obviously has \relT\ (because $A \subseteq H$). Combining this with the fact that $(G,H,A^m)$ has \relT\ (with approximation, by \cref{NearInvariantForNormal}), we conclude that $(G,H,A)$ has \relT, as desired.
 \end{proof}
 
 \begin{rem} \label{TripleNilpHasTMoreGen}
 The proof of \cref{TripleNilpHasT} applies somewhat more generally than is specified in the statement of the \lcnamecref{TripleNilpHasT}. More precisely, after the assumption that $HA$ is closed, it suffices to make the following two additional assumptions:
 	\begin{enumerate}
	\item For every finite-dimensional, unitary representation $\rho$ of~$M$, the closure of $\rho(M)$ has an abelian subgroup of finite index. (For example, this is true when $M$ is virtually solvable, and also when $M$ is a connected Lie group whose Levi subgroup has no compact factors.)
	\item For every finite-index, closed subgroup~$M'$ of~$M$, if $m$ is the index of $(M')^{(1)} \cap A$ in~$A$, then $m < \infty$, and there is a compact subset~$C$ of~$H$, such that $A \subseteq C \cdot A^m$.
	\end{enumerate}

Also note that every closed subgroup of a compactly generated nilpotent group is compactly generated \cite[Thm.~6, p.~38]{Platonov}. Therefore, if $M$~is nilpotent, then it would suffice to assume $M$ is compactly generated, instead of assuming that $A$~is compactly generated. 
\end{rem}

\section{Proofs of results stated in the Introduction} \label{IntroProofSect}

In this \lcnamecref{IntroProofSect}, we prove that all of the results stated in the Introduction are consequences of \cref{TripleHasT}. 
We first prove \cref{relT(GH1meas)}, \cref{relT(GH1)}, \cref{Nilpotent/N1}, and \cref{StrongT,KazhdanSetForNilp} (while mentioning an additional \lcnamecref{H/H^1cpct} and \lcnamecref{NnotNormal} along the way). These are followed by \cref{SemiNilpotent}, which is a special case of $(\ref{SemiEquiv-pair} \Leftrightarrow \ref{SemiEquiv-Amen})$ of \cref{SemiEquiv} below. 

Recall that, as stated in \cref{SecondCountable}, all locally compact groups are assumed to be second countable, and therefore $\sigma$-compact.

\begin{proof}[\bf Proof of \cref{relT(GH1meas)}]
This is the special case of \cref{TripleHasT} in which $G = H$.
\end{proof}

\begin{proof}[\bf Proof of \cref{relT(GH1)}]
This is a special case of \cref{relT(GH1meas)} (see \fullcref{WhenInvMeas}{center}).
\end{proof}

The following immediate consequence of \cref{relT(GH1)} is a generalization of \cite[Cor.~3.5.3, p.~177]{BHV} (which is the special case where $G = H$).

\begin{cor} \label{H/H^1cpct}
Let $H$ be a closed subgroup of a locally compact group~$G$, and let $A$ be a closed subgroup of the center of~$H$, such that $A$ is normal in~$G$. If $(G/A, H/A)$ has \relT, and $H/H^{(1)}$ is compact, then $(G, H)$ has \relT. 
\end{cor}

\begin{proof}[\bf Proof of \cref{Nilpotent/N1}]
The direction ($\Rightarrow$) is obvious. The other direction follows easily from \cref{relT(GH1)} (or,  if the reader prefers, from \cref{relT(GH1meas)}, \ref{TripleHasT}, or~\ref{TripleNilpHasT}), by induction on the nilpotence class of~$N$.
\end{proof}

\begin{rem} \label{NnotNormal}
In the statement of \cref{Nilpotent/N1}, the assumption that $N$ is normal can be weakened slightly, to the assumption that $N$ has a central series $N = N_0 \triangleright N_1 \triangleright N_2 \triangleright \cdots \triangleright N_c = \{e\}$, such that $N_i \triangleleft G$, for all~$i > 0$. (The subgroup $N = N_0$ does not need to be normal in~$G$.)
\end{rem}

\begin{proof}[\bf Proof of \cref{StrongT}]
Since ($\Rightarrow$) is obvious, we prove only ($\Leftarrow$). Let $A = Z(N) \cap N^{(1)}$. By induction on the nilpotence class of~$N$, we may assume that $\bigl( H \ltimes (N/A), H, N/A \bigr)$ has \relT. Then the desired conclusion is immediate from \cref{TripleHasT}, by letting $G = H \ltimes N$ and $M = N$.
\end{proof}

\begin{proof}[\bf Proof of \cref{KazhdanSetForNilp}]
Let $G = G_0 \triangleright G_1 \triangleright G_2 \triangleright \cdots \triangleright G_k = \{e\}$ be:
	\begin{itemize}
	\item the descending central series of~$G$, if $G$ is nilpotent,
	or
	\item the derived series of~$G$, if $G$ is a connected, real split, solvable Lie group.
	\end{itemize}
In either case, we have $G_k = \{e\}$ for some~$k$. By induction on~$k$, we may assume the image of~$Q$ in $G/G_{k-1}$ is a Kazhdan set for $G/G_{k-1}$. Applying \cref{WhenInvMeas,KazhdanSetForG/A} (with $A = G_{k-1}$) tells us that $Q$ is a Kazhdan set for the pair $(G,G_1)$. By combining this with the fact that the image of~$Q$ in $G/G_1 = G^{ab}$ is a Kazhdan set for~$G/G_1$, we conclude that $Q$ is a Kazhdan set for~$G$.
\end{proof}

Our next goal is \cref{SemiEquiv}, which is an extension of \cref{SemiNilpotent} that also incorporates \cref{StrongT}. Its proof uses the following result.
%
%


\begin{prop}[Raja {\cite[Lem.~3.1]{Raja}}] \label{SemiTripleAmen}
Let $H$ be a locally compact group that acts on a $1$-connected, abelian Lie group~$N$.
Then $(H \ltimes N, H, N)$ has \relT\ if and only if there does not exist a closed, connected, $H$-invariant, proper subgroup~$L$ of~$N$, such that $\Int_{N/L}(H)^\bullet$ is amenable.
\end{prop}

Although our main interest is in groups that are locally compact, we state the following result without this assumption on~$H$:

\begin{cor} \label{SemiEquiv}
If a topological group~$H$ acts on a $1$-connected, nilpotent Lie group~$N$, then the following are equivalent:
	\begin{enumerate}
	\item \label{SemiEquiv-pair}
The pair $(H \ltimes N, N)$ has \relT.
	\item \label{SemiEquiv-pairAbel}
The pair $(H \ltimes N^{ab}, N^{ab})$ has \relT.
	\item \label{SemiEquiv-triple}
The triple $(H \ltimes N, H, N)$ has \relT.
	\item \label{SemiEquiv-tripleAbel}
The triple $(H \ltimes N^{ab}, H, N^{ab})$ has \relT.
	\item \label{SemiEquiv-Amen}
For every closed, connected, $H$-invariant, proper subgroup~$L$ of~$N$ that contains~$N^{(1)}$, the group $\Int_{N/L}(H)^\bullet$ is \underline{not} amenable. 
	\item \label{SemiEquiv-FiniteSet}
	There exists a \underline{finite} subset $Q$ of~$H$, and $\epsilon > 0$, such that if $\pi$~is any unitary representation of $H \ltimes N$ that has nonzero $(Q,\epsilon)$-invariant vectors, then $\pi$~has nonzero $N$-invariant vectors.
	\end{enumerate}
\end{cor}

\begin{proof}
It is easy to establish 
	($\ref{SemiEquiv-triple} \Rightarrow \ref{SemiEquiv-pair}  \Rightarrow \ref{SemiEquiv-pairAbel}$)
	and 
	$(\ref{SemiEquiv-FiniteSet} \Rightarrow \ref{SemiEquiv-triple} \Rightarrow \ref{SemiEquiv-tripleAbel}  \Rightarrow \ref{SemiEquiv-pairAbel}$). 
Also, ($\ref{SemiEquiv-pairAbel} \Rightarrow \ref{SemiEquiv-Amen}$)%
\refnote{AmenQuotNotT}
 is well known \cite[Prop.~2.2 (ii'~$\Rightarrow$~i), p.~391]{CornulierValette}.

Therefore, it suffices to prove ($\ref{SemiEquiv-Amen} \Rightarrow \ref{SemiEquiv-FiniteSet}$).
Let $H^\bullet$ be the closure of the image of~$H$ in $\Aut(N^{ab})$. 
 Now, as $H^\bullet$ is a closed subgroup of the Lie group $\Aut(N^{ab})$, it is separable.
Therefore, there is a countable subgroup~$\Gamma$ of~$H$ whose image is dense in~$H^\bullet$. 
Every $\Gamma$-invariant, closed subgroup of~$N^{ab}$ is also $H^\bullet$-invariant, and is therefore $H$-invariant. If we give $\Gamma$ the discrete topology, then $\Gamma$ is locally compact, so we see from \cref{SemiTripleAmen} that $(\Gamma \ltimes N^{ab}, \Gamma, N^{ab})$ has \relT. 
Then \cref{StrongT} tells us that  $(\Gamma \ltimes N, \Gamma, N)$ has \relT. 

By a standard argument \cite[Thm.~1.2 (a2~$\Rightarrow$~a1)]{Jolissaint}, this implies there is a finite  subset $Q$ of~$\Gamma$, and $\epsilon > 0$, such that if $\pi$~is any unitary representation of $\Gamma \ltimes N$ that has nonzero $(Q,\epsilon)$-invariant vectors, then $\pi$~has nonzero $N$-invariant vectors.
 Since every unitary representation of $H \ltimes N$ restricts to a (continuous) unitary representation of $\Gamma \ltimes N$, this completes the proof.
\end{proof}

\begin{rem}
It is obvious that if a triple $(G,H,N)$ has \relT, then the pair $(G,N)$ also has \relT. The converse does not hold in general. (For example, it is easy to see that if $N$ is infinite and discrete, then the triple $(H \ltimes N, H, N)$ never has \relT\ \cite[Rem.~2.1.8]{Jaudon}. But the pair $(H \ltimes N, N)$ may have \relT.) \Cref{SemiEquiv} shows that the converse does hold when $N$ is a $1$-connected, nilpotent Lie group (and $G = H \ltimes N$).
\end{rem}

Furthermore, the equivalence of \pref{SemiEquiv-pair} and \pref{SemiEquiv-pairAbel} in \cref{SemiEquiv} does not require $N$ to be a Lie group:

\begin{cor} \label{SemiEquivNotLie}
If a topological group~$H$ acts on a compactly generated, locally compact, nilpotent group~$N$, then the following are equivalent:
	\begin{enumerate}
	\item \label{SemiEquivNotLie-pair}
The pair $(H \ltimes N, N)$ has \relT.
	\item \label{SemiEquivNotLie-pairAbel}
The pair $(H \ltimes N^{ab}, N^{ab})$ has \relT.
	\end{enumerate}
\end{cor}

\begin{proof} 
It suffices to prove $(\ref{SemiEquivNotLie-pairAbel} \Rightarrow \ref{SemiEquivNotLie-pair})$. Since the maximal compact subgroup of~$N$ is unique \cseebelow{NilpHasCpct}, 
it is normal in $H \ltimes N$, so there is no harm in modding it out. Therefore, we may assume that $N$ has no nontrivial compact subgroups, so $N$ is a (nilpotent) Lie group \cseebelow{open/cpct=Lie}, 
such that $N^\circ$~is $1$-connected \cseebelow{NoCpctSubgrps}
and $N/N^\circ$ is finitely generated (because $N$ is compactly generated) and torsion-free \cseebelow{NoCpctSubgrps}.
Then $N$ can be embedded as a closed, cocompact subgroup of a $1$-connected Lie group~$N_1$ \cite[Thm.~2.20, p.~42]{Raghunathan}. Every automorphism of~$N$ extends uniquely to an automorphism of~$N_1$ \cite[Cor.~1 on p.~34]{Raghunathan}, so we may form the semidirect product $H \ltimes N_1$, which contains $H \ltimes N$ as a closed, cocompact subgroup.
Since $(H \ltimes N^{ab}, N^{ab})$ has \relT, and $N$ is cocompact in $N_1$, it is easy to see that $(H \ltimes N_1^{ab}, N_1^{ab})$ has%
\refnote{N1abRelT}
 \relT.
 Namely, let $\rho \colon H \ltimes N^{ab} \to H \ltimes N_1^{ab}$ be the natural homomorphism, and let $N^\bullet$ be the closure of $\rho(N^{ab})$.  Since $(H \ltimes N^{ab},N^{ab})$ has \relT, we know that $(H \ltimes N_1^{ab},N^\bullet)$ has \relT.  Also, $N_1^{ab}/N^\bullet$ is compact (because $N_1/N$ is compact), so $\bigl( H \ltimes (N_1^{ab}/N^\bullet), N_1^{ab}/N^\bullet \bigr)$ has \relT. Therefore $(H \ltimes N_1^{ab}, N_1^{ab})$ has \relT.

Now \fullcref{SemiEquiv}{FiniteSet} provides a finitely generated subgroup~$\Gamma$ of~$H$, such that $(\Gamma \ltimes N_1, N_1)$ has \relT\ (where $\Gamma$ is given the discrete topology).
Since $N_1/N$ is compact, and $N_1$~is nilpotent, there is a unique $N_1$-invariant probability measure on $N_1/N$. The uniqueness implies that the measure is $\Gamma$-invariant, so we obtain a ($\Gamma \ltimes N_1$)-invariant probability measure on $(\Gamma \ltimes N_1)/(\Gamma \ltimes N)$. Therefore, since $(\Gamma \ltimes N_1, N_1)$ has \relT, we conclude from \cite[Prop.~2.4(1)]{Iozzi} that $(\Gamma \ltimes N, N)$ has%
\refnote{IozziProp}
 \relT.

 Since every unitary representation of $H \ltimes N$ restricts to a (continuous) unitary representation of $\Gamma \ltimes N$, this implies that $(H \ltimes N, N)$ has \relT.
\end{proof}

\section{The largest subgroup with relative Property~\textup(T\textup)} \label{LargestRelTSect}

It is easy to construct examples in which $(G,L_1)$, $(G,L_2)$, \dots, $(G,L_k)$ have \relT, but if we let $L$ be the subgroup generated by $L_1 \cup L_2 \cup \cdots \cup L_k$, then $(G, L)$ does not have \relT.
(For example, let $G$ be a simple Lie group that does not have Property~(T), and let $L_1,L_2,\ldots,L_k$ be compact subgroups that generate~$G$.) \Cref{LargestRelT} provides a situation in which this pathology does not arise. 
The proof does not require the main results proved in \cref{IntroProofSect}, but it does use several basic facts about locally compact groups and \relT.


\begin{lem}[cf.\ {\cite[Thm.~2]{Losert1} and \cite[Lem.~3.1]{Dani-convolution}}]\label{NilpHasCpct}
Every%
\refnote{NilpHasCpctPf}
compactly generated, locally compact, nilpotent group~$N$ has a unique maximal compact subgroup.
\end{lem}


\begin{notation}
We use $G^\circ$ for the identity component of the topological group~$G$.
\end{notation}

\begin{thm}[{}{\cite[Theorem on p.~175 (and Lem.~2.3.1, p.~54)]{MontgomeryZippin}}] \label{open/cpct=Lie}
Let $G$ be a locally compact group.Then some open subgroup~$H$ of~$G$ has a compact, normal subgroup~$C$, such that $H/C$ is a Lie group. 

Moreover, for an appropriate choice of~$H$, the compact subgroup~$C$ can be chosen to be contained in any neighborhood of the identity in~$H$.

Furthermore, if $G/G^\circ$ is compact, then we may take $H = G$.
\end{thm}

The following consequence is well known.

\begin{cor} \label{NoCpctSubgrps}
Let%
\refnote{NoCpctSubgrpsPf}
$G$ be a locally compact group. If $G/G^\circ$ is compact, and $G$~has no nontrivial, compact subgroups, then $G$ is a $1$-connected Lie group.
\end{cor}


\begin{lem}[cf.\ {\cite[Lem.~2.2]{Shalom-Algebraization}}] \label{HinFiniteProduct}
Assume $H,H_1,\ldots,H_n$ are closed subgroups of a locally compact group~$G$, and $C$ is a compact subset of~$G$, such that $C H_1 H_2 \cdots H_n$ contains~$H$. If the pair $(G, H_i)$ has \relT, for each~$i$, then $(G,H)$ has \relT.
\end{lem}

\begin{proof}
Let $\pi$ be a unitary representation of~$G$ on a Hilbert space~$\hilbert$, such that $\pi$ has almost-invariant vectors, and let $\delta = 4^{-(n+2)}$.
For each~$i$, \cref{NearInvariantForPair} provides a compact subset $Q_i$ of~$G$ and $\epsilon_i > 0$, such that if $\eta$~is any $(Q_i, \epsilon_i)$-invariant unit vector, then there is an $H_i$-invariant unit vector $\eta_i$, such that $\| \eta - \eta_i \| < \delta/2$. 
Now, let $\eta$ be a $(Q,\epsilon)$-invariant unit vector, where $Q = C \cup \bigcup_{i=1}^n Q_i$ and $\epsilon = \min(\delta, \epsilon_1,\ldots,\epsilon_n)$. Then
	$$ \text{$\| \pi(g) \eta - \eta \| < \delta$, for all $g \in H_1 \cup \cdots H_n \cup C$} .$$
This implies $\| \pi(h) \eta - \eta \| < 1/2$ for all $h \in H$.%
\refnote{FinProdClose}
So $\pi$ has a nonzero $H$-invariant vector \cite[Lem.~2.2]{Jolissaint}.
\end{proof}

\begin{lem} \label{FiniteProduct}
Let $H$ be a compactly generated, locally compact, nilpotent group, and let $\mathcal{L}$ be a collection of closed subgroups of~$H$ that generates a dense subgroup of~$H$. Then, for some~$n$, there exist $L_1,\ldots,L_n \in \mathcal{L}$, and a compact subset~$C$ of~$H$, such that $C L_1 L_2 \cdots L_n = H$.
\end{lem}

\begin{proof}
We may assume that $H$ is a Lie group with no nontrivial compact subgroups (by modding out the maximal compact subgroup \csee{NilpHasCpct,open/cpct=Lie}).
Let $H^{(k)}$ be the closure of the last nontrivial term of the descending central series of~$H$.
The desired conclusion is easy if~$H$ is abelian (and therefore isomorphic to $\RR^m \times \ZZ^n$ for some $m$ and~$n$), so we may assume $H^{(k)} \neq H$. By induction on the nilpotence class of~$H$, we may assume that there is a finite product $X = L_1 \cdots L_n$ of subgroups in~$\mathcal{L}$, and a compact subset~$C$ of~$H$, such that $C H^{(k)} X = H$. Note that, for each $g \in H^{(k-1)}$, the map $x \mapsto [x,g]$ is a homomorphism from $H/H^{(k)}$ to $H^{(k)}$. Since $\dim H^{(k)} + \rank H^{(k)}/(H^{(k)})^\circ$ is finite, there is a finite subset $\{g_1,\ldots,g_m\}$ of $H^{(k-1)}$, such that $\prod_{i=1}^m [H,g_i]$ is dense in a cocompact subgroup of the abelian group~$H^{(k)}$. 
So there is a compact subset $C_0$ of~$H^{(k)}$, such that $C_0 \cdot \prod_{i=1}^m [H,g_i] = H^{(k)}$.
Since $H^{(k)}$ is abelian, this implies that $C_0 \cdot \prod_{i=1}^m [C, g_i] \cdot \prod_{i=1}^m [X,g_i] = H^{(k)}$. Then
	\begin{align*}
	C \cdot C_0 \cdot \prod_{i=1}^m [C, g_i]  \cdot \prod_{i=1}^m (X \cdot g_i X g_i^{-1}) \cdot X
	=
	C H^{(k)} X = H
	. & \qedhere \end{align*}
\end{proof}

\begin{cor} \label{LargestRelT}
Let $N$ be a closed, compactly generated, nilpotent, normal subgroup of a locally compact group~$G$, and let $\mathcal{T}$ be the collection of all subgroups~$L$ of~$N$, such that $(G,L)$ has \relT. Then $\mathcal{T}$ has a unique largest element~$L^\dagger$, and $L^\dagger$ is a closed, normal subgroup of~$G$.
\end{cor}

\begin{proof}
Let $L^\dagger$ be the closure of the subgroup generated by the subgroups in~$\mathcal{T}$. 
 \Cref{FiniteProduct} tells us there is a product $L_1 L_2 \cdots L_n$ of finitely many elements of~$\mathcal{T}$, and a compact set~$C$, such that $C L_1 L_2 \cdots L_n = L^\dagger$. Since $(G,L_i)$ has \relT, for each~$i$, it follows from \cref{HinFiniteProduct} that $(G, L^\dagger)$ has \relT. So $L^\dagger \in \mathcal{T}$. By definition, $L^\dagger$ contains every element of~$\mathcal{T}$, so this implies that $L^\dagger$ is the unique largest element of~$\mathcal{T}$.
 
 Furthermore, $L^\dagger$ is closed by definition. Also, if $L$ is any conjugate of~$L^\dagger$, then $(G,L)$ has \relT\  (because $(G,L^\dagger)$ has \relT), so $L \subseteq L^\dagger$. Therefore, $L^\dagger$ is normal.
\end{proof}

We also have the following weaker conclusion without the assumption that $G$ is locally compact.

\begin{cor}\label{LargestRelTNormal}
Let $N$ be a closed, locally compact, compactly generated, nilpotent, normal subgroup of a topological group~$G$, and let $\mathcal{T_\triangleleft}$ be the collection of all subgroups~$L$ of~$N$, such that $(G,L)$ has \relT\ and $L \triangleleft G$. Then $\mathcal{T_\triangleleft}$ has a unique largest element~$L^\dagger_\triangleleft$, and $L^\dagger_\triangleleft$ is a closed, normal subgroup of~$G$.
\end{cor}

\begin{proof}
Nothing in the proof of \cref{LargestRelT} relies on the assumption that $G$ is locally compact, other than the application of \cref{NearInvariantForPair} in the proof of \cref{HinFiniteProduct}. Although \cref{NearInvariantForPair} may not be true for general topological groups, its conclusion holds when $H \triangleleft G$, by the same standard argument that is used in the proof of \cref{NearInvariantForNormal}.
\end{proof}

\section{Relative Property (T) and amenability} \label{GeneralizeSemiNilpotentSect}

The main result in this \lcnamecref{LargestRelTSect} is \cref{N/LisAmen}, which provides additional information about the subgroup~$L^\dagger$ of \cref{LargestRelT} (under a connectivity assumption on~$N$). 
This implies \cref{IffAmenAlmConn}, which is a generalization of \cref{SemiNilpotent} that does not require the subgroup~$N$ to be a Lie group. 
The statements of these results require the following extension of \cref{IntOnAbel} to this setting:

\begin{notation}
Suppose $N$ and~$L$ are closed, normal subgroups of a locally compact group~$G$, such that $L \subseteq N$, and $N/L$ is a Lie group.
	\begin{enumerate}
	\item $\Int_{N/L} \colon G\to \Aut(N/L)$ is the natural map defined by the action of~$G$ on $N/L$ by conjugation.
	\item We use $\Int_{N/L}(G)^\bullet$ to denote the \emph{closure} of the image of this homomorphism.
	\end{enumerate}
\end{notation}

\begin{cor} \label{N/LisAmen}
Let $N$ be a closed, locally compact, nilpotent, normal subgroup of a topological group~$G$, such that $N/N^\circ$ is compact. Then $N$ has a unique largest subgroup~$L^\dagger$, such that $(G,L^\dagger)$ has \relT. Furthermore, $L^\dagger$~is a closed, normal subgroup of~$G$, the quotient $N/L^\dagger$~is a $1$-connected Lie group, and $\Int_{N/L^\dagger}(G)^\bullet$ is amenable.
\end{cor}


The following consequence of ($\ref{SemiEquiv-Amen} \Rightarrow \ref{SemiEquiv-pairAbel}$) of \cref{SemiEquiv} is essentially the special case of \cref{N/LisAmen} in which $N$~is a connected, abelian Lie group. It will be the basis of a proof by induction.

\begin{cor}[cf.\ {\cite[Prop.~2.2 (i~$\Rightarrow$~ii$'$)]{CornulierValette} and \cite[Cor.~3.2]{Raja}}] \label{Abel/LisAmen}
Assume a topological group~$H$ acts on a connected, abelian Lie group~$N$. Then $N$ contains a closed, connected, $H$-invariant, normal subgroup~$L$,%
\refnote{Abel/LisAmenPf}
such that $(H \ltimes N,L)$ has \relT, and $\Int_{N/L}(H)^\bullet$ is amenable.
\end{cor}

The following elementary observation can reduce problems about arbitrary normal subgroups to the easier case of semidirect products.

\begin{lem} \label{SemiToPair}
Let $N$ be a closed, normal subgroup of a topological group~$G$, and let $L$ be a subgroup of~$N$. Form the semidirect product $G \ltimes N$, where $G$ acts on~$N$ by conjugation. If\/ $(G \ltimes N, L)$ has \relT, then $(G,L)$ also has \relT.
\end{lem}

\begin{proof}
For any unitary representation $\pi$ of~$G$, there is a unitary representation $\pi'$ of~$G \ltimes N$ that is defined by $\pi'(g, n) = \pi(g) \, \pi(n)$, for $g \in G$ and $n \in N$. If $\pi$ has almost-invariant vectors, then so does~$\pi'$. 
\end{proof}

\begin{rem}
The converse of \cref{SemiToPair} is not true. For example, let $G$ be a Lie group with Kazhdan's property~(T), such that $Z(G)$ is not compact, and let $N = L = Z(G)$. (In particular, $G$ could be the universal cover of $\mathrm{Sp}(4,\RR)$ \cite[Example~1.7.13(ii), p.~67]{BHV}.) Then $(G,N)$ has \relT\ (in fact, $(G,G)$ has \relT), but $G \ltimes N \cong G \times N$, and $(G \times N, N)$ does not have \relT\ (because $N$ is a noncompact, abelian Lie group, and therefore does not have Kazhdan's property~(T)).
\end{rem}

As the final preparation for the proof of \Cref{N/LisAmen}, we establish one more \lcnamecref{L+N/LisAmen}:

\begin{lem}[cf.\ {\cite[Lem.~2.3(i)]{CornulierValette}}] \label{L+N/LisAmen}
Suppose $N$ and~$L$ are closed, normal subgroups of a locally compact group~$G$, such that $L \subseteq N$. Assume that $L$ is connected, and that $N$ is a $1$-connected, nilpotent Lie group. If\/ $\Int_{N/L}(G)^\bullet$ and $\Int_{L}(G)^\bullet$ are amenable, then $\Int_N(G)^\bullet$ is amenable.
\end{lem}

\begin{proof}
Let $\Lie N$ and $\Lie L$ be the Lie algebras of $N$ and~$L$, respectively, and let 
	$$P = \{\, T \in \GL(\Lie N) \mid T(\Lie L) = \Lie L \,\} .$$
Since $N$ is a $1$-connected, nilpotent Lie group, we can identify $\Aut(N)$ with $\Aut(\Lie N)$, which is a closed subgroup of~$P$. It is well known that, by choosing a complement~$W$ to~$\Lie L$ in~$\Lie N$, we have $P = \bigl( \GL(W) \times \GL(\Lie L) \bigr) \ltimes K$, where $K$ is the kernel of the natural map $P \to \GL(\Lie N / \Lie L) \times \GL(\Lie L)$. Therefore, 
	$$\Int_N(G)^\bullet \subseteq \bigl( \Int_{N/L}(G)^\bullet \times \Int_{L}(G)^\bullet \bigr) \ltimes K .$$
Since $K$ is abelian (and hence amenable), we conclude that $\Int_N(G)^\bullet$ is a closed subgroup of an amenable group, and is therefore amenable.
\end{proof}

\begin{proof}[\bf Proof of Corollary \ref{N/LisAmen}]
By modding out~$L^\dagger_\triangleleft$ \csee{LargestRelTNormal}, there is no harm in assuming that it is trivial, which means:
\begin{equation} \tag{$*$} \label{noL}
\begin{matrix}
\text{$N$ does not contain any nontrivial normal subgroup~$L$ of~$G$,} \\
\text{such that $(G,L)$ has \relT.}
\end{matrix}
\end{equation}

Since $N$ is a compactly generated, nilpotent group, it has a unique maximal compact subgroup~$K$ \csee{NilpHasCpct}. Then $(G,K)$ has \relT\ (since~$K$ is compact), and $K \triangleleft G$ (because of the uniqueness), so~\pref{noL} implies that $K$~is trivial.
This means that $N$ has no nontrivial compact subgroups. So $N$ is a $1$-connected Lie group \csee{NoCpctSubgrps} (and, by assumption, $N$~is nilpotent).

All that remains is to show that $\Int_N(G)^\bullet$ is amenable. 
By \cref{SemiToPair}, we may assume that $G$ is a semidirect product $H \ltimes N$.
Let $A = N^{(1)} \cap Z(N)$. (Note that $A$ is a closed, connected, abelian, normal subgroup of~$G$.)
By induction on the rank of~$N$ (and \cref{Abel/LisAmen} for the base case where $A$ is trivial, so $N$~is abelian), we may assume that $N/A$ contains a closed, normal subgroup~$L'/A$ of~$G/A$, such that $(G/A,L'/A)$ has \relT, $N/L'$ is $1$-connected, and $\Int_{N/L'}(G)^\bullet$ is amenable. From \cref{relT(GH1)}, we conclude that $\bigl( G, (L')^{(1)} \bigr)$ has \relT. So \eqref{noL} tells us that $(L')^{(1)}$ is trivial, which means $L'$ is abelian. 
Also, since $N$~is connected and $N/L'$ is $1$-connected, we know that $L'$ is connected.
Therefore, we may apply \cref{Abel/LisAmen} to the semidirect product $G \ltimes L'$ (and compare with \pref{noL}), to conclude that $\Int_{L'}(G)^\bullet$ is amenable. We now know that $\Int_{N/L'}(G)^\bullet$ and $\Int_{L'}(G)^\bullet$ are amenable, so \cref{L+N/LisAmen} tells us that $\Int_N(G)^\bullet$ is amenable, as desired.
\end{proof}

\begin{rem} 
Assume $G$, $N$, and~$L^\dagger$ are as in \cref{N/LisAmen}.
	\begin{enumerate}
	\item
	 If $N$ is connected, then $L^\dagger$ is also connected. To see this from the proof of \cref{N/LisAmen}, it suffices to note that the maximal subgroup~$K$ must be connected, since $N$ is homeomorphic to $K \times \RR^n$. (This is well known for connected Lie groups and, in fact, was proved by Iwasawa \cite[Thm.~13, p.~549]{Iwasawa} for connected groups that are approximated by Lie groups. \Cref{open/cpct=Lie} implies that every connected, locally compact group can be so approximated.)
	\item If $G$ is a semidirect product $H \ltimes N$, then the subgroup~$L^\dagger$ either is compact, or projects nontrivially into $N/N^{(1)}$. To establish this, assume, without loss of generality, that $N$ is a $1$-connected Lie group (by modding out the maximal compact subgroup). If $\Int_N(H)^\bullet$ is amenable, then $L^\dagger$ is trivial.\refnote{AmenQuotNotTNilp}
	 (This is an easy generalization of ($\ref{SemiEquiv-pairAbel} \Rightarrow \ref{SemiEquiv-Amen}$) of \cref{SemiEquiv}.)
	If not, then the Zariski closure of $\Int_N(H)^\bullet$ has a noncompact, semisimple subgroup~$M$. Since $M$ acts nontrivially on the $1$-connected, nilpotent group~$N$, it is well known that $M$ must act nontrivially on~$N/N^{(1)}$.%
	\refnote{NontrivOnAbel} 
So $\Int_{N/N^{(1)}}(G)^\bullet$ is not amenable. Therefore, \cref{N/LisAmen} tells us that $L^\dagger$~is not contained in~$N^{(1)}$.
	\end{enumerate}
\end{rem}

\begin{rem}
The assumption that $N/N^\circ$ is compact cannot be deleted from the statement of \cref{N/LisAmen}.
(So a connectivity assumption is also necessary in \cref{Abel/LisAmen}.)
Here is a counterexample that is adapted from the proof of \cite[Prop.~2.2 (ii$'$~$\Rightarrow$~i)]{CornulierValette}.

Let 
	$\alpha = \sqrt{2}$, 
	$\ints = \ZZ[\alpha]$,  
	$Q(x_1,x_2,x_3) = x_1^2 + x_2^2 + \alpha x_3^2$,
	$\Gamma = \SO_3(Q; \ints)$,
	$N = \ints^3 \cong \ZZ^6$,
	$G = \Gamma \ltimes N$,
	and 
	$H = \SO(3) \ltimes \RR^3$,
and let $\pi$ be the unitary representation of~$G$ obtained by composing the homomorphism
	\begin{align*}
	G
	&= \Gamma \ltimes N
	\hookrightarrow \SO(3) \ltimes \RR^3
	= H
	\end{align*}
with the left-regular representation of~$H$. Since $H$ is amenable, $\pi$ has almost-invariant vectors. However, if $L$ is any nontrivial subgroup of~$N$, then the image of~$L$ in~$\RR^3$ is noncompact, so $\pi$ has no nonzero $L$-invariant vector. Therefore, $L$~must be trivial if $(G,L)$ has \relT.

However, Restriction of Scalars embeds~$G$ as a lattice in $\bigl( \SO(3) \ltimes \RR^3 \bigr) \times \bigl( \SO(2,1) \ltimes \RR^3 \bigr)$ (cf.\ \cite[\S5.5]{Morris}), so it is clear that $\Int_N(G)^\bullet$ is not amenable.
\end{rem}

\Cref{N/LisAmen} implies the following generalization of $(\ref{SemiEquiv-pair} \Leftrightarrow \ref{SemiEquiv-Amen})$ of \cref{SemiEquiv}:

\begin{cor} \label{IffAmenAlmConn}
Assume a topological group~$H$ acts on a locally compact, nilpotent group~$N$, such that $N/N^\circ$ is compact. Then $(H \ltimes N,N)$ has \relT\ if and only if there does \underline{not} exist a closed, $H$-invariant, normal subgroup~$L$ of~$N$, such that $N/L$~is a nontrivial, $1$-connected Lie group, and $\Int_{N/L}(H)^\bullet$ is amenable.
\end{cor}

\begin{rem} \label{IffAmenCpctlyGen}
\Cref{IffAmenAlmConn} can be used to determine whether $(H \ltimes N,N)$ has \relT, even if we replace the assumption that $N/N^\circ$ is compact with the weaker assumption that $N$~is compactly generated. 
This is because the argument in the first paragraph 
of the proof of \cref{SemiEquivNotLie} constructs a $1$-connected, nilpotent Lie group~$N_1$, such that $(H \ltimes N,N)$ has \relT\ if and only if $(H \ltimes N_1,N_1)$ has \relT.
\end{rem}

\section{Relative Property (T) for subsets} \label{SubsetSect}

As was mentioned in \cref{GenToSubsets}, Y.~Cornulier \cite[p.~302]{Cornulier} has generalized the notion of \relT\ to pairs $(G,H)$ in which $H$ is a sub\emph{set} of~$G$, rather than a sub\emph{group}. We propose the following natural analogue for triples $(G,H,M)$ in which $M$ is a subset:

\begin{defn} \label{RelTTripleSubsetDefn}
Assume $H$ is a closed subgroup of a topological group~$G$, and $M$~is a subset of~$G$. 
	\begin{enumerate}
	\item \label{RelTTripleSubsetDefn-alsohas}
	To say that the triple $(G,H,M)$ has \emph{\relT} means that for every $\epsilon > 0$, there exist a compact subset~$Q$ of~$H$ and $\delta > 0$, such that every unitary representation of~$G$ with nonzero $(Q,\delta)$-invariant vectors also has nonzero $(M,\epsilon)$-invariant vectors.
	\item \label{RelTTripleSubsetDefn-approx}
	To say that the triple $(G,H,M)$ has \emph{\relT\ with approximation} means that for every $\epsilon > 0$, there exist a compact subset~$Q$ of~$H$ and $\delta > 0$, such that every $(Q,\delta)$-invariant vector for any unitary representation of~$G$ is also $(M,\epsilon)$-invariant.
	\end{enumerate}
\end{defn}

\begin{rem}
If $H = G$, and $G$~is locally compact, then it follows from the proof of \cite[Thm.~2.2.3 ($1 \Leftrightarrow 2$)]{Cornulier} that \pref{RelTTripleSubsetDefn-alsohas} and~\pref{RelTTripleSubsetDefn-approx} of \cref{RelTTripleSubsetDefn} are equivalent\refnote{RelTSubset}
to each other, and to Cornulier's definition of \relT\ for the pair $(G,M)$. Also, it is easy to see that \cref{RelTTripleSubsetDefn} is consistent with \cref{RelTTriple,RelTApproxTriple} (because being $(M,\epsilon)$-invariant is equivalent to being close to an $M$-invariant vector in the situation where $M$ is a closed subgroup of~$G$ \cite[Lem.~2.2]{Jolissaint}).
\end{rem}

\begin{notation}
 For elements $m_1$ and~$m_2$ of any group, we let 
	$$[m_1,m_2] = m_1^{-1} m_2^{-1} m_1 m_2 .$$
\end{notation}

In our discussion of \relT\ for subsets, the following trivial observation replaces \cref{MMinv} as a way to obtain almost-invariant vectors for~$\pi$ from almost-invariant vectors for $\pi \otimes_{\dual A} \conj\pi$.

\begin{lem} \label{MAlmInvSubset}
Assume
	$\rho$ is a unitary representation of a topological group~$G$ on a Hilbert space~$\hilbert$, 
	$m_1,m_2 \in G$, 
	$\xi \in \hilbert$,
	$\xi' = \xi \otimes \conj{\xi} \in \hilbert \otimes \conj{\hilbert}$,
	and 
	$\rho' = \rho \otimes \conj{\rho}$.
Then
	$$ \| \rho \bigl( [m_1,m_2] \bigr) \xi - \xi \| 
	\le 2 \bigl( \| \rho'(m_1)\xi' - \xi' \| + \| \rho'(m_2)\xi' - \xi' \| \bigr) .$$
\end{lem}

\begin{proof}
For convenience, let $\epsilon_i = \| \rho'(m_i) \xi' - \xi' \|$ (for $i = 1,2$), and assume, without loss of generality, that $\|\xi\| = 1$. For $i = 1,2$, there exists a unique $\lambda_i \in \CC$ (with $|\lambda_i| \le 1$), such that 
 	$$\rho(m_i)\xi - \lambda_i \xi \ \perp \ \xi .$$
To avoid some uncomfortably long expressions in the following sentence, let $v_i = \rho(m_i)\xi - \lambda_i \xi$, so $v_i \perp \xi$. Then the three vectors
	$$ \text{$v_i \otimes \conj{\rho(m_i)\xi}$,
	\quad
	$\lambda_i \xi \otimes \conj{v_i}$,
	\quad and \quad
	$\xi'$}
	$$
are pairwise orthogonal, so the Pythagorean Theorem tells us
	\begin{align*}
	\bigl\| \rho'(m_i)\xi' - \xi' \bigr\|^2 
	&= \bigl\| v_i \otimes \conj{\rho(m_i)\xi}\bigr\|^2
	+ \bigl\| \lambda_i \xi \otimes \conj{v_i} \bigr\|^2
	+ \bigl\| \bigl( |\lambda_i|^2 - 1 \bigr) \xi' \bigr\|^2
	\\&\ge \bigl\| v_i \otimes \conj{\rho(m_i)\xi}\bigr\|^2
	\\&= \| v_i \|^2
	, \end{align*}
which means 
	$$ \epsilon_i \ge \| \rho(m_i)\xi - \lambda_i \xi \| .$$
Therefore
	\begin{align*}
	 \| \rho( m_1 m_2) \xi - \lambda_1 \lambda_2 \xi \|
	 &\le \bigl\| \rho(m_1) \bigl( \rho(m_2) \xi - \lambda_2 \xi \bigr) \bigr\| + \| \rho( m_1) (\lambda_2 \xi) - \lambda_1 \lambda_2 \xi \|
	 \\&= \| \rho(m_2) \xi - \lambda_2 \xi \| + |\lambda_2| \cdot \| \rho( m_1) \xi - \lambda_1 \xi \|
	 \\&\le \| \rho(m_2) \xi - \lambda_2 \xi \| + \| \rho( m_1) \xi - \lambda_1 \xi \|
	 \\&\le \epsilon_2 + \epsilon_1
	 . \end{align*}
Since the same is true after interchanging the subscripts $1$ and~$2$ (and $\lambda_1 \lambda_2 = \lambda_2 \lambda_1$), we conclude that
	\begin{align*}
	\bigl\| \rho \bigl( [m_1, m_2] \bigr) \xi -  \xi \bigr\| 
	&= \| \rho( m_1 m_2) \xi -  \rho( m_2 m_1) \xi \| 
	\\&\le \| \rho( m_1 m_2) \xi - \lambda_1 \lambda_2 \xi \| + \| \rho( m_2 m_1) \xi - \lambda_1 \lambda_2 \xi \|
	\\&\le (\epsilon_2 + \epsilon_1) + (\epsilon_1 + \epsilon_2)
	\\&= 2(\epsilon_1 + \epsilon_2)
	.  \qedhere \end{align*}
\end{proof}

The following \lcnamecref{TripleHasTSubset} is an analogue of \cref{TripleHasT} that does not require the set~$M$ to be a subgroup.

\begin{defn} \ 
	\begin{enumerate}
	\item A \emph{probability measure} is a finite measure that has been normalized to have total mass~$1$.
	\item We use the total variation norm $\| \cdot \|$ to provide a metric on the space of probability measures (on any topological space). 
	\item For a subset~$M$ of a group~$G$, we let 
	$$\comm{M}{M} = \{\, [m_1,m_2] \mid m_1,m_2 \in M \,\} .$$
	\end{enumerate}
\end{defn}

\begin{thm} \label{TripleHasTSubset}
Let $H$ be a closed subgroup of a locally compact group~$G$,  let $A$ be a closed, abelian, normal subgroup of~$G$, and let $M$~be a subset of~$G$. For every $\epsilon' > 0$, assume there exists $\delta' > 0$, such that if $\mu$ is any $(M,\delta')$-invariant probability measure on~$\dual A$ that is quasi-invariant for the action of~$G$, then $\mu(\dual A^M) \ge 1 - \epsilon'$, where $\dual A^M$ is the set of fixed points of~$M$. If $HA$~is closed and\/ $(G/A, HA/A, MA/A)$ has \relT\ with approximation, then\/ $\bigl( G, H, \comm{M}{M} \bigr)$ has \relT\ with approximation. 
\end{thm}

\begin{proof}
This is adapted from the proof of \cref{TripleHasT}.
Given an arbitrary $\epsilon > 0$, 
choose $\epsilon' > 0$ small enough that if $\xi$ is any $\epsilon'$-invariant unit vector, and $\| \xi - \eta\|^2 < \epsilon'$, then $\eta$~is $\epsilon/4$-invariant.
Also, let $\delta'$ be a value that corresponds to this value of~$\epsilon'$ in the assumption in the statement of the \lcnamecref{TripleHasTSubset} (and assume $\delta' < \epsilon'/2$). 
Since $(G/A, HA/A, M/A)$ has \relT\ with approximation (and $HA$ is closed), there exist a compact subset~$Q$ of~$H$ and $\delta > 0$, such that if $\xi$ is any $(Q,\delta)$-invariant vector for any unitary representation of $G/A$, then $\xi$~is $(M,\delta'/2)$-invariant.

Now, suppose $\pi$ is a unitary representation of~$G$, such that $\pi$ has a nonzero $(Q,\delta/3)$-invariant vector~$f$. (We wish to show that $f$ is $\bigl( [M,M] ,\epsilon \bigr)$-invariant.) 
By replacing $\pi$ with the direct sum $\pi \oplus \pi \oplus \cdots$ of infinitely many copies of itself, we may assume that all irreducible representations appearing in the direct integral decomposition of~$\pi|_A$ have the same multiplicity (namely,~$\infty$). By definition, this means that $\pi|_A$ is homogeneous, so \cref{FiberSect} provides a quasi-invariant probability measure~$\mu$ on~$\dual A$, a Borel cocycle $\alpha \colon G \times \dual A \to \unitary(\hilbert)$, a corresponding realization of~$\pi$ as a representation on the Hilbert space $L^2(\dual A, \mu;\hilbert)$, and a unitary representation $\pi' = \pi \otimes_{\dual A} \conj{\pi}$ of~$G$. 

Since $\pi$ has been realized as a representation on $L^2( \dual A, \mu ; \hilbert)$, we know that $f \in L^2( \dual A, \mu ; \hilbert)$. Then \cref{TensorAlmInv} provides a $(Q,\delta)$-invariant unit vector~$f'$ for~$\pi'$. By the choice of $Q$ and~$\delta$ (and Remark \ref{Akerpip}), we know that 
	$$ \text{$f'$~is $(M,\delta'/2)$-invariant} .$$
Then it is straightforward to check that $\|f'\|^2 \, \mu$ is an $(M,\delta')$-invariant\refnote{InvtMeasFromInvtVector} 
measure on~$\dual A$.
Also, by perturbing $f$ slightly, we could assume that it is nonzero almost everywhere, so $f'$ is also nonzero almost everywhere. Then $\|f'\|^2 \, \mu$ is quasi-invariant for the $G$-action (because $\mu$ is quasi-invariant). By the choice of~$\delta'$, this implies that $\int_{\dual A^M} \|f'\|^2 \, d\mu > 1 - \epsilon'$ (assuming that $f'$ has been normalized to be a unit vector in~$L^2$).
\Cref{TensorAlmInv} tells us 
	$$f'(\lambda) = \frac{1}{\|f(\lambda)\|} \, f(\lambda) \otimes \conj{f(\lambda)} ,$$
so $\| f'(\lambda) \| = \| f(\lambda) \|$ for all $\lambda \in \dual A$. This implies
	$$\int_{\dual A^M} \|f\|^2 \, d\mu = \int_{\dual A^M} \|f'\|^2 \, d\mu > 1 - \epsilon' ,$$
so $\bigl\| f|_{\dual A^M} - f \bigr\| < \epsilon'$. This means that $f$ is well approximated by a function that is supported on~$\dual A^M$, so, to simplify the argument, we will assume that $f$ itself is supported on~$\dual A^M$.

For each fixed $\lambda\in\dual A^M$, the function $\alpha(m,\lambda)$ is a representation $\rho_\lambda$ of $M$ on~$\hilbert$, so $M$ acts on $L^2 ( \dual A^M, \mu ; \hilbert \otimes \conj{\hilbert} )$ by
$$ \bigl( \pi'(m)f' \bigr)(\lambda)= (\rho_\lambda \otimes \conj{\rho_\lambda})(m) f'(\lambda).$$
Therefore, we see from \cref{MAlmInvSubset} that
	$$ \bigl\| \pi\bigl([m_1,m_2]\bigr)f(\lambda) - f(\lambda) \bigr\|
	\le 2 \bigl( \| \pi'(m_1)f'(\lambda) - f'(\lambda) \|
		+ \| \pi'(m_2)f' (\lambda) - f'(\lambda) \| \bigr).$$
Since this is true for all~$\lambda$, we conclude that
	\begin{align*}
	\bigl\| \pi\bigl([m_1,m_2]\bigr)f  - f \bigr\|_2
	&\le 2 \bigl( \| \pi'(m_1)f' - f' \|_2
		+ \| \pi'(m_2)f' - f' \|_2 \bigr) 
	\\&\le 2\left( \frac{\delta'}{2} + \frac{\delta'}{2} \right)
	< \epsilon
	. \qedhere \end{align*}
\end{proof}

\Cref{TripleHasTSubset} immediately implies the following natural generalization of \cref{relT(GH1)}, in which $H$ is a subset, rather than a subgroup. (We do not need to mention approximation, because Jolissaint's proof of \cref{NearInvariantForPair} shows that \relT\ is equivalent to \relT\ with approximation\refnote{ApproxForPair}
for all pairs $(G,H)$, even if $H$ is only a subset, not a subgroup.)
 
\begin{cor} \label{relT(GH1)Subset}
Let $A$ be a closed, abelian, normal subgroup of a locally compact group~$G$, and $H$ be a subset of~$G$ that centralizes~$A$. If $(G/A, HA/A)$ has \relT, then $\bigl(G, \comm{H}{H} \bigr)$ has \relT. 
\end{cor}

\section{Other observations about \relT} \label{OtherSect}

We close the paper with some tangential observations about \relT.

\subsection{Relative Property (T) for connected, normal, Lie subgroups}

If $N$ is a connected Lie group, then, since the group $U/U_S$ in the following \lcnamecref{NormalLie} is a connected, nilpotent Lie group, \cref{IffAmenAlmConn} determines whether or not $(H \ltimes N, N)$ has \relT, without the need to assume $N$ is nilpotent.

\begin{prop} \label{NormalLie}
Suppose a locally compact group~$H$ acts on a connected Lie group~$N$. Let $S$ be the closure of the product of the noncompact simple factors of a Levi subgroup of~$N$, let $U$ be the nilradical of~$N$, and let 
	$$ U_S = \cl{ \vphantom{\bigl(} [S,U] \cdot (S \cap U)} . $$ 
Then $(H \ltimes N, N)$ has \relT\ if and only if:
	\begin{enumerate}
	\item \label{NormalLie-S}
	$S$ has Kazhdan's Property~\textup(T\textup),
	\item \label{NormalLie-R}
	$N/\cl{SU}$ is compact,
	and
	\item \label{NormalLie-U}
	$\bigl( H \ltimes (U/U_S), U/U_S \bigr)$ has \relT. 
	\end{enumerate}
\end{prop}

\begin{proof}

($\Rightarrow$) 
The adjoint group $\Ad S$ is a quotient of~$N$, 
so $( H \ltimes \Ad S , \Ad S )$ has \relT. 
However, $\Ad S$ is a connected, semisimple Lie group, so its outer automorphism group is finite. Therefore, we may assume that $H$ acts on $\Ad S$ by inner automorphisms (after replacing $H$ with a finite-index subgroup). This implies that 
	$$ H \ltimes \Ad S = C_{H \ltimes \Ad S}(S) \cdot \Ad S \cong H \times \Ad S .$$
So we now know that $( H \times \Ad S , \Ad S )$ has \relT. This implies that $\Ad S$ has Kazhdan's Property~\textup(T\textup). Then~\pref{NormalLie-S} follows from \cref{relT(GH1)} (or the special case proved by J.--P.~Serre that is mentioned in \cref{relT(GH1)Serre}).

By definition, $S$ is contained in the closure of some Levi subgroup~$S^+$ of~$N$.  Since the pair $\bigl( H \ltimes N/(S^+U), N/(S^+U) \bigr)$ has \relT, and the structure theory of Lie groups tells us that $\Aut(N)$ acts on $N/(S^+U)$ via a finite group\refnote{actR/Ufinite}, we see that $N/(S^+U)$ is compact. Since $S^+/S$ is compact (by the definition of~$S$), this implies~\pref{NormalLie-R}. 

Since $(H \ltimes N, N)$ has \relT, \pref{NormalLie-R} implies that $\bigl( H \ltimes SU , SU \bigr)$ has \relT\ (see \cite[Cor.~4.1(2)]{Jolissaint}). Passing to a quotient yields~\pref{NormalLie-U}.

\medbreak

($\Leftarrow$) Suppose $\pi$ is a unitary representation of $H \ltimes N$ that has almost-invariant vectors.
We wish to show that $N$ has invariant vectors. 
By induction on $\dim N$ (and \cref{NearInvariantForPair}), we may assume that no nontrivial, connected, $H$-invariant, normal subgroup of~$N$ has nonzero invariant vectors.\refnote{NormalLie-nonormalAid}

Therefore, \cref{LargestRelT} implies that there is no nontrivial, connected subgroup~$U_0$ of~$U$, such that $(H \ltimes N,U_0)$ has \relT.
 So $U$ has no nontrivial, compact subgroups, and is therefore $1$-connected \csee{NoCpctSubgrps}. Then \cref{N/LisAmen} implies that $S$ centralizes~$U$. So $S \triangleleft H \ltimes N$.\refnote{NormalLie-Snormal}
However, we see from~\pref{NormalLie-S} that $(H \ltimes N, S)$ has \relT. So the assumption of the previous paragraph implies that $S$ is trivial. Then $U_S$ is obviously also trivial. So $U/U_S = U = SU$.
Therefore, \pref{NormalLie-U}~tells us that $\bigl( H \ltimes (SU), SU \bigr)$ has \relT. 
Then the assumption of the previous paragraph implies that $SU$ is trivial.
Therefore, \pref{NormalLie-R} tells us that $N$ is compact, so $( H \ltimes N, N)$ has \relT, as desired.
\end{proof}


\begin{rem} 
The proof of \cref{NormalLie} yields the following general result (which is slightly weaker than \cref{NormalLie} in the special case where $G$ is a semidirect product):

\begin{narrower} \it\noindent Suppose $N$ is a closed, normal subgroup of a locally compact group~$G$, such that $N$ is a connected Lie group. Define $S$, $U$, and~$U_S$ as in \cref{NormalLie}.
Then $(G, N)$ has \relT\ if and only if:
	\begin{enumerate} \leftskip = 2\parindent
	\item $S$ has Kazhdan's Property~\textup(T\textup),
	\item $N/\cl{SU}$ is compact,
	and
	\item $( G / U_S, U /U_S )$ has \relT.
	\end{enumerate}
\end{narrower}
Stronger results were already known in the special case where $G$ is a connected Lie group. (See, for example, \cite[Cor.~3.3.2]{Cornulier}.)
\end{rem}

\subsection{Relative Property (T) for solvable subgroups}

In the statement of \cref{Nilpotent/N1}, the assumption that $N$ is nilpotent cannot be replaced with the weaker assumption that $N$ is solvable. (For example, let $G = N$ be a noncompact, solvable group, such that $G/G^{(1)}$ is compact.) However, it would suffice to assume that $N$ is a connected, real split, solvable Lie group (see \fullcref{WhenInvMeas}{split}). Also, we have the following easy consequence of \cref{Nilpotent/N1} that applies to some other solvable groups.

\begin{notation}
If $N$ is a locally compact group, then 
	$$N^{(2)}= (N^{(1)})^{(1)} = \cl{\bigl[ [N,N], [N,N] \bigr]}$$
is the closure of the second derived group of~$N$.
\end{notation}

\begin{cor} \label{Solvable/N2}
Let $N$ be a closed, normal subgroup of a locally compact group~$G$, such that $N^{(1)}$ is nilpotent.
 Then $(G,N)$ has \relT\ if and only if $(G/N^{(2)}, N/N^{(2)})$ has \relT. 
 \end{cor}
 
 For example, every virtually polycyclic group has a (characteristic) finite-index subgroup whose commutator subgroup is nilpotent (see \cite[Cor.~4.11, p.~59]{Raghunathan}). Here is another example:

\begin{cor} \label{ConnSolv}
Let $N$ be a connected, closed, solvable, normal subgroup of a locally compact group~$G$. Then $(G,N)$ has \relT\ if and only if $(G/N^{(2)}, N/N^{(2)})$ has \relT.
 \end{cor}

\begin{proof}
It is well known that the assumptions on~$N$ imply that $N^{(1)}$ has a unique maximal compact subgroup~$C_1$, and that $N^{(1)}/C_1$ is nilpotent.%
\refnote{N1MaxCpct}
So the desired conclusion is obtained by applying \cref{Solvable/N2} to the pair $(G/C_1,N/C_1)$.
\end{proof}

\subsection{Homomorphisms with a dense image}

The following \lcnamecref{TIffDense} generalizes a result of Y.~Cornulier and R.~Tessera \cite[Cor.~2]{CornulierTessera}.

\begin{cor} \label{TIffDense}
Let $H$, $N$, and~$H_1$ be locally compact groups. Assume $N$~is nilpotent, $H$ acts on~$N$, and we are given a homomorphism $H_1 \to H$ with dense image.  Then $(H \ltimes N,N)$ has \relT\ if and only if $(H_1 \ltimes N,N)$ has \relT.

Moreover, if $(H \ltimes N,N)$ has \relT, then there is a finitely generated group~$\Gamma$ and a homomorphism $\Gamma \to H$, such that $(\Gamma \ltimes N, N)$ has \relT.
\end{cor}

\begin{proof}
In the case where $N$ is abelian, this is \cite[Cor.~2]{CornulierTessera}.
The general case follows from this by applying \cref{Nilpotent/N1}.
\end{proof}

If we assume that $N$ is a connected Lie group, then the assumption that $N$ is nilpotent can be eliminated:

\begin{cor} \label{TIffDenseLie}
Let $H$ and~$H_1$ be locally compact groups, and let $N$ be a connected Lie group. Assume $H$ acts on~$N$, and we are given a homomorphism $H_1 \to H$ with dense image.  Then $(H \ltimes N,N)$ has \relT\ if and only if $(H_1 \ltimes N,N)$ has \relT.

Moreover, if $(H \ltimes N,N)$ has \relT, then there is a finitely generated group~$\Gamma$ and a homomorphism $\Gamma \to H$, such that $(\Gamma \ltimes N, N)$ has \relT.
\end{cor}

\begin{proof}
\Cref{NormalLie} reduces the problem to the case where $N$ is nilpotent, which is handled by \cref{TIffDense}.
\end{proof}

If we assume that $N$ is a $1$-connected Lie group, then \cref{TIffDense} can be extended to triples, and does not require $H$ or~$H_1$ to be locally compact:

\begin{cor} 
Let $H$ and~$H_1$ be topological groups, and let $N$ be a $1$-connected, nilpotent Lie group. Assume $H$ acts on~$N$, and we are given a homomorphism $H_1 \to H$ with dense image.  Then the following are equivalent:
	\begin{enumerate}
	\item $(H \ltimes N, N)$ has \relT.
	\item $(H_1 \ltimes N, N)$ has \relT.
	\item $(H \ltimes N,H, N)$ has \relT.
	\item $(H_1 \ltimes N,H_1, N)$ has \relT.
	\end{enumerate}

Moreover, if these conditions hold, then there is a finitely generated group~$\Gamma$ and a homomorphism $\Gamma \to H$, such that $(\Gamma \ltimes N, N)$ and $(\Gamma \ltimes N, \Gamma, N)$ have \relT.
\end{cor}

\begin{proof}
A closed subgroup of~$N$ is $H$-invariant if and only if it is $H_1$-invariant, so the equivalence of the four conditions follows from \cref{SemiEquiv}. The final conclusion follows from \fullcref{SemiEquiv}{FiniteSet}.
\end{proof}

\subsection{An observation on the center}

Although we are mostly interested in the abelianization of a nilpotent subgroup~$H$, we also record the following observation regarding the opposite end of a central series of~$H$.

\begin{cor} \label{CenterT}
Let $H$ be a nilpotent subgroup of a locally compact group~$G$. If there is a nontrivial subgroup~$L$ of~$H$, such that $(G,L)$ has \relT, then there is a nontrivial subgroup~$Z$ of the center of~$H$, such that $(G,Z)$ has \relT.
\end{cor}

\begin{proof}
Consider the ascending central series of~$H$: 
	$$\{e\} = Z_0 \subset Z_1 \subset \cdots \subset Z_c = H .$$
Let $k$ be minimal, such that $(G,L)$ has \relT, for some closed, nontrivial subgroup~$L$ of~$Z_k$. We may assume $L \not\subseteq Z(H)$, so there is some~$i$ (which we choose to be minimal), such that $[L, Z_i] \neq \{e\}$.  Choose $h \in Z_i$, such that $[L,h] \neq \{e\}$.
Then $[L,h] \subseteq [H, Z_i] \subseteq Z_{i-1}$, so the minimality of~$i$ implies that $[L,h]$ centralizes~$L$. Therefore $[\ell_1\ell_2,h] = [\ell_1,h] \, [\ell_2,h]$ for $\ell_1,\ell_2 \in L$, so $[L,h]$ is a subgroup. Also, $[L,h] \subseteq  L \cdot h L h^{-1}$  has \relT\ by \cref{HinFiniteProduct}. Since $\{e\} \neq [L,h] \subseteq [Z_k,H] \subseteq Z_{k-1}$, this contradicts the minimality of~$k$.
\end{proof}

\begin{rem}
The proof of \cref{CenterT} establishes the following general fact about subgroups of a nilpotent group: If $L$ is a nontrivial subgroup of a nilpotent group~$H$, then there exist finitely many conjugates of~$L$, such that the product of these conjugates contains a nontrivial subgroup of the center of~$H$.
\end{rem}

\Addresses

\newpage

\gdef\SAuthor{\emph{Notes to aid the referee}}
\gdef\STitle{\emph{Notes to aid the referee}}

\begin{appendix}

\numberwithin{equation}{aid}

\section{Notes to aid the referee}

\begin{aid} \label{FuncAnal-TensorContPf}
\begin{proof}[\bf Proof of \pref{FuncAnal-TensorCont} of \cref{FuncAnal}]
Suppose $(U_i,V_i) \to (U,V)$. For $\xi \in \hilbert_1 \otimes \hilbert_2$, we need to show that $(U_i \otimes V_i)\xi \to (U \otimes V)\xi$. 

Fix $\epsilon > 0$. The vector $\xi$ is approximated to within~$\epsilon$ by a finite sum $\xi' = \sum c_j u_j \otimes v_j$. For each~$j$, we have $U_i u_j \to U u_j$ and $V_i v_j \to V v_j$ as $i \to \infty$. Since $\{u_j\} \cup \{v_j\}$ is finite, this implies $(U_i,V_i)\xi' \to (U,V)\xi'$, so 
	\begin{align*}
	\limsup \| (U_i \otimes V_i)\xi - (U \otimes V)\xi \| < 2\epsilon
	. & \qedhere \end{align*}
\end{proof}
\end{aid}

\begin{aid} \label{AinKernel}
The formula $\pi(a) = \int_{\dual A} \lambda(a) \, d \mkern0.4\thinmuskip \PP(\lambda)$ means that 
	$$ \text{$\bigl( \pi(a) f \bigr)(\lambda) = \lambda(a) \, f(\lambda)$ \ for $a \in A$ and $f \in L^2(\dual A, \mu ; \hilbert)$} . $$
Also, $A$ obviously acts trivially on~$\dual A$, so $a^{-1} \lambda = \lambda$ and $D(g,\lambda)$ is identically~$1$. Therefore, by comparing with the displayed equation characterizing $\alpha(g,\lambda)$, we see that $\alpha(a,\lambda) = \lambda(a)$. Then 
	$$ \alpha(a,\lambda) \otimes \conj{\alpha(a,\lambda)} = \lambda(a) \cdot \conj{\lambda(a)} = 1 ,$$
so we have 
	$$ \bigl( \pi'(a) f \bigr)(\lambda) = \sqrt{D(a,\lambda)} \, \bigl( \alpha(a,\lambda) \otimes \conj{\alpha(a,\lambda)} \bigr) \, f(a^{-1}\lambda)
	= \sqrt{1} \cdot 1 \cdot f(\lambda) = f(\lambda) .$$
Therefore $a$ is in the kernel of~$\pi'$.
\end{aid}

\begin{aid} \label{BDT}
\cite[Cor.~1.3]{Dani-quasi} tells us that if $m \in M$ and $v$~is any point in the support of a finite $M$-invariant measure~$\mu$ on the vector space~$\dual A$, then the ($\mathop{\mathrm{Ad}} m$)-orbit of~$v$ is bounded. This implies that $v$ is in the span of the eigenspaces of $\mathop{\mathrm{Ad}} m$ corresponding to eigenvalues (in~$\CC$) of absolute value~$1$. However, since $M$ is real split, $1$ is the only eigenvalue of $\mathop{\mathrm{Ad}} m$ that has absolute value~$1$. So $v$ is in the $1$-eigenspace of~$\mathop{\mathrm{Ad}} m$. In other words, $v$~is fixed by $\mathop{\mathrm{Ad}} m$. Since $m$ is an arbitrary element of~$M$, and $v$~is an arbitrary point in the support of~$\mu$, we conclude that the support of~$\mu$ is contained in the set of fixed points of~$M$.
\end{aid}

\begin{aid} \label{InvtMeasToL2}
Assume $\mu$ is $(H,\delta')$-invariant. 
For any fixed $h \in H$, let
	$$ \text{$X^> = \{\, \lambda \in \dual A \mid D(h,\lambda) > 1\,\}$
	\ and \ 
	$X^< = \{\, \lambda \in \dual A \mid D(h,\lambda) < 1\,\}$} .$$
Then
	\begin{align*}
	\| \rho(h) 1 - 1 \|_2^2
	&= \int_{\dual A} | \sqrt{D(h,\lambda)} - 1 |^2 \, d\mu
	\\&= \int_{\dual A} | \sqrt{D(h,\lambda)} - 1 | \cdot | \sqrt{D(h,\lambda)} - 1 |\, d\mu
	\\&\le \int_{\dual A} | \sqrt{D(h,\lambda)} - 1 | \cdot | \sqrt{D(h,\lambda)} + 1 |\, d\mu
	\\&= \int_{\dual A} | D(h,\lambda) - 1 | \, d\mu
	\\&= \int_{X^>} \bigl( D(h,\lambda) - 1 \bigr) \, d\mu \ + \  \int_{X^<} \bigl( 1 - D(h,\lambda) \bigr) \, d\mu
	\\&= \bigl( (h^*\mu)(X^>) - \mu(X^>) \bigr) + \bigl( \mu(X^<) - (h^*\mu)(X^<) \bigr)
	\\&\le 2 \| h^*\mu - \mu \|
	\\& \le 2 \delta'
	, \end{align*}
so $1$ is $(H, \sqrt{2\delta'} \bigr)$-invariant in $L^2(\dual A, \mu)$. Therefore, we may let $\delta' = (\epsilon'')^2/32$.
\end{aid}


\begin{aid} \label{AmenQuotNotT}
Suppose there is a closed, connected, $H$-invariant, proper subgroup~$L$ of~$N$ that contains~$N^{(1)}$, such that $\Int_{N/L}(H)^\bullet$ is amenable. For convenience, let $H^\bullet = \Int_{N/L}(H)^\bullet$ and $A = N/L$, so $H^\bullet$ is amenable and $A$~is a noncompact, abelian Lie group. Let $\rho$ be the unitary representation of $H \ltimes N^{ab}$ that is obtained by composing the natural homomorphism $H \ltimes N^{ab} \to H^\bullet \ltimes A$ with the regular representation of $H^\bullet \ltimes A$. Since $H^\bullet \ltimes A$ is amenable, this representation has almost-invariant vectors. However, it cannot have $N^{ab}$-invariant vectors, because $A$~is noncompact, and therefore does not fix any nonzero vectors in the regular representation of $H^\bullet \ltimes A$. So $(H \ltimes N^{ab}, N^{ab})$ does not have \relT.
\end{aid}

 \begin{aid} \label{N1abRelT}
  (\emph{Warning:} The kernel of~$\rho$ is $(N \cap N_1^{(1)})/N^{(1)}$, which may be nontrivial.) 
 Consider a unitary representation~$\pi$ of $H \ltimes N_1^{ab}$ that has almost-invariant vectors. 
  We obtain a representation of~$H \ltimes N^{ab}$ by composing $\pi$ with~$\rho$.
Since $(H \ltimes N^{ab}, N^{ab})$ has \relT, the representation $\pi \circ \rho$ must have nonzero $N^{ab}$-invariant vectors, which means that $\pi$ has $\rho(N^{ab})$-invariant vectors. The space of $\rho(N^{ab})$-invariant vectors  is ($H \ltimes N_1^{ab}$)-invariant (because $\rho(N^{ab})$ is normal in $H \ltimes N_1^{ab}$) and has almost-invariant vectors \ccf{NearInvariantForPair}. 

There is a compact subset~$C$ of~$N_1^{ab}$, such that $N_1^{ab} = C \, \rho(N^{ab})$ (because $N_1/N$ is compact). For any $\epsilon > 0$, there is a nonzero $(C,\epsilon)$-invariant vector in the space of $\rho(N^{ab})$-invariant vectors. Any such vector is $(N_1^{ab}, \epsilon)$-invariant. Since we may take $\epsilon < 1$ (and $N_1^{ab}$ is a subgroup), we conclude that $\pi$ has nonzero $N_1^{ab}$-invariant vectors \cite[Lem.~2.2]{Jolissaint}. Therefore $(H \ltimes N_1^{ab}, N_1^{ab})$ has \relT.
 \end{aid}

 \begin{aid} \label{IozziProp}
 Prop.~2.4(1) of \cite{Iozzi} states that if $G$ is a locally compact group, $A$ is a normal subgroup of~$G$, and $L$ is a  closed subgroup of~$G$, such that $(G, A)$ has \relT\ and $G/L$ has a finite $G$-invariant measure, then the pair $(L, L \cap A)$ also has \relT.
 
 We take $G = \Gamma \ltimes N_1$, $L = \Gamma \ltimes N$, and $A = N_1$ (so $L \cap A = N$).
 \end{aid}

\begin{aid} \label{NilpHasCpctPf}
This observation must be well known, and follows easily from results in the literature (such as by combining \cite[Thm.~2]{Losert1} with \cite[Lem.~3.1]{Dani-convolution}). However, we do not know where to find an explicit proof (for general locally compact groups, rather than merely Lie groups), so here is a proof.

\begin{proof}
By induction on the nilpotence class of~$N$, we may assume that $N^{(1)}$ has a unique maximal compact subgroup. By modding this out, we may assume that $N^{(1)}$ has no nontrivial compact subgroups.

Now, let $Z$ be the center of~$N$. 
By induction on the nilpotence class of~$N$, we may assume that the quotient $N/Z$ has a unique maximal compact subgroup $C/Z$. Since $C$ contains every compact subgroup of~$N$ (and the comment in the last paragraph of \cref{TripleNilpHasTMoreGen} tells us that $C$ is compactly generated), there is no harm in assuming $C = N$, so $N/Z$ is compact.

Now, it suffices to show that $N$ is abelian (because it is well known that every compactly generated, locally compact, abelian group has a unique maximal compact subgroup 
\cite[Thm.~23.11(a), p.~197]{Stroppel}). 
Suppose $N$ is not abelian. Then, since $N$ is nilpotent, there exists $g \in N \smallsetminus Z$, such that $[g,N] \subseteq Z$. Since $[g,N] \subseteq Z$, the map $x \mapsto [g,x]$ is a homomorphism. The kernel of this homomorphism contains $Z$, and $N/Z$ is compact, so the image $[g,N]$ is a compact subgroup of $N^{(1)}$. However, we said in the first paragraph that $N^{(1)}$ has no nontrivial compact subgroups, so this implies that $[g,N] = \{e\}$, which contradicts the fact that $g \notin Z$.
\end{proof}

\vbox{

}

\end{aid}

\begin{aid} \label{NoCpctSubgrpsPf}
\begin{proof}
Since $G/G^\circ$ is compact, \cref{open/cpct=Lie} provides a compact subgroup~$C$ of~$G$, such that $G/C$ is a Lie group. By assumption, $C$ must be trivial, so $G$~itself is a Lie group. Since $G/G^\circ$ is compact, this implies that $G/G^\circ$ is finite (since the identity component of a Lie group is an open subgroup). However, it is a general fact about Lie groups with finitely many connected components that $G$ is diffeomorphic to $K \times \RR^n$, for some maximal compact subgroup~$K$ and some $n \in \NN$ \cite[Thm.~14.3.11, p.~544 (and Thm.~14.1.3, p.~531)]{HilgertNeeb}. 
In our case, $K$~is trivial, so $G$ itself is diffeomorphic to~$\RR^n$, and is therefore $1$-connected.
\end{proof}
\end{aid}


\begin{aid} \label{FinProdClose}
For convenience, we reverse the numbering of the subgroups $H_1, \ldots,H_n$, so we may write $h = h_{n+1} h_n h_{n-1} \cdots h_1$ where $h_{i+1} \in C$ and $h_i \in H_i$ for $i \le n$. Let $g_i = h_{i} h_{i-1} \cdots h_1$. Then, by induction on~$k$, we have
	\begin{align*}
	\| \pi(g_k) \eta - \eta \|
	&\le \sum_{i = 0}^{k-1} \| \pi(h_{i+1} g_i)\eta - \pi(g_i) \eta \|
	\\&\le \sum_{i = 0}^{k-1} \Bigl( \| \pi(h_{i+1}) \bigl( \pi(g_i)\eta - \eta \bigr) \| + \| \pi(h_{i+1})\eta - \eta \| + \| \pi(g_i)\eta - \eta \| \Bigr)
	\\&= \sum_{i = 0}^{k-1} \Bigl( \| \pi(h_{i+1})\eta - \eta \| + 2\| \pi(g_i)\eta - \eta \| \Bigr)
	\\&< \sum_{i = 0}^{k-1} ( \delta + 2 \cdot 4^i \delta)
	\\&\le 4^k \delta
	.\end{align*}
Letting $k = n+1$ tells us that $\| \pi(h) \eta - \eta \| < 4^{n+1} \delta = 1/4$.
\end{aid}

\begin{aid} \label{Abel/LisAmenPf}
\begin{proof}
By modding out the maximal compact subgroup of~$N$, we may assume that $N$ is a $1$-connected, abelian Lie group \csee{NilpHasCpct,open/cpct=Lie}. So we may identify $N$ with the vector space~$\RR^n$. 

Let $L$ be the (unique) largest $H$-invariant subspace of~$\RR^n$, such that $(H \ltimes \RR^n,L)$ has \relT. By modding out~$L$, we may assume there is no nontrivial $H$-invariant subspace~$M$ of~$\RR^n$, such that $(H \ltimes \RR^n,M)$ has \relT. 

We wish to show $\Int_{\RR^n}(H)^\bullet$ is amenable.
Suppose not. Then the Zariski closure of $\Int_{\RR^n}(H)^\bullet$ contains a noncompact simple subgroup~$S$. Let $M_0$ be a nonzero subspace on which $S$ acts irreducibly (and nontrivially), and let $M$ be the smallest $H$-invariant subspace that contains~$M_0$. If $L$ is any proper $H$-invariant subspace of~$M$, then $L$ cannot contain~$M_0$, so $S$ acts nontrivially on $M/L$, so $\Int_{M/L}(H)^\bullet$ is not amenable. 

Now, by applying ($\ref{SemiEquiv-Amen} \Rightarrow \ref{SemiEquiv-pairAbel}$) of \cref{SemiEquiv}, we see that $(H \ltimes N,M)$ has \relT, which is a contradiction. 
\end{proof}

\medskip

Alternatively, the proof of \cite[Prop.~2.2 (i~$\Rightarrow$~ii$'$)]{CornulierValette} easily generalizes to this setting. 

(However, the proof of \cite[Cor.~3.2]{Raja} has a gap. Namely, if $\Int_{\RR^n}(H)^\bullet$ is not amenable, then the proof shows there are $H$-invariant subspaces $L \subsetneq M \subseteq \RR^n$, such that $\bigl( H \ltimes (\RR^n/L), M/L )$ has \relT, but the proof does not explain why it is possible to choose $L$ to be $\{0\}$.)
\end{aid}

\begin{aid} \label{AmenQuotNotTNilp}
(This is a slight modification of \cref{AmenQuotNotT} or the proof of \cite[Prop.~2.2 (ii'~$\Rightarrow$~i), p.~391]{CornulierValette}.)
For convenience, let $H^\bullet = \Int_N(H)^\bullet$, so $H^\bullet$ is amenable. Let $\rho$ be the unitary representation of $G = H \ltimes N$ that is obtained by composing the natural homomorphism $H \ltimes N \to H^\bullet \ltimes N$ with the regular representation of $H^\bullet \ltimes N$. Since $H^\bullet \ltimes N$ is amenable, this representation has almost-invariant vectors. However, for any closed, noncompact subgroup~$L$ of~$N$, the representation~$\rho$ cannot have $L$-invariant vectors, because the stabilizer of any vector in the regular representation of $H^\bullet \ltimes N$ is compact. So $(H \ltimes N, L)$ does not have \relT.
\end{aid}

\begin{aid}  \label{NontrivOnAbel}
The following fact is well known.

\begin{lem*}
Let $M$ be a nontrivial, semisimple group of automorphisms of a nilpotent Lie algebra~$\Lie N$. 
Then $M$ acts nontrivially on the abelianization $\Lie N/[\Lie N,\Lie N]$.
\end{lem*}

\begin{proof}
Let $\mathcal C$ be the centralizer of~$M$ in~$\Lie N$. Since $M$ is nontrivial, the subalgebra~$\mathcal C$~is proper, so it is contained in a maximal subalgebra~$\mathcal M$ of~$\Lie N$. Since every (finite-dimensional) representation of a semisimple Lie algebra is completely reducible, there is an $M$-invariant complement~$\mathcal W$ to~$\mathcal M$ in~$\Lie N$.

Suppose $M$ acts trivially on $\Lie N/[\Lie N,\Lie N]$. Then, for all $w \in \mathcal W$ and $h \in M$, we have $hw \in w + [\Lie N,\Lie N]$, so $w - hw \in [\Lie N,\Lie N] \subseteq \mathcal M$, because every maximal subalgebra of a nilpotent Lie algebra contains the commutator subalgebra \cite[Cor.~2, p.~420]{Marshall}. However, $hw \in \mathcal W$ (because $\mathcal W$ is $M$-invariant), so we also have $w - hw \in \mathcal W$. Therefore $w - hw \in \mathcal M \cap \mathcal W = \{0\}$. This implies $w \in \mathcal C \subseteq \mathcal M$. So $w = 0$ (since $\mathcal M \cap \mathcal W = \{0\}$). However, $w$~is an arbitrary element of~$\mathcal W$, so this is a contradiction.
\end{proof}

\vbox{

}
\end{aid}

\begin{aid} \label{RelTSubset}
It is obvious that if the triple $(G,G,M)$ has \relT\ with approximation, then it has \relT.

If $(G,G,M)$ has \relT, then, for every $\epsilon > 0$, and for every unitary representation~$\pi$ of~$G$ that has almost-invariant vectors, $\pi$~has $(M,\epsilon)$-invariant vectors. So \cite[Thm.~2.2.3 ($2 \Rightarrow 1$)]{Cornulier} tells us that $(G,M)$ has \relT\ (according to Cornulier's definition).

Finally, we use the proof of \cite[Thm.~2.2.3 ($1 \Rightarrow 2$)]{Cornulier} to show that if $(G,M)$ has \relT, then $(G,G,M)$ has \relT\ with approximation. Suppose not. Then there exists $\epsilon > 0$, such that, for every compact subset~$Q$ of~$G$ and every $\delta > 0$, there exists a $(Q,\delta)$-invariant vector~$\xi_{Q,\delta}$ for some unitary representation of~$G$, such that $\xi_{Q,\delta}$~is not $(M,\epsilon)$-invariant. Let $\varphi_{Q,\delta}$ be the matrix coefficient corresponding to~$\xi$. Then, when $Q$ gets large and $\delta \to 0$, the matrix coefficient $\varphi_{Q,\delta}$ converges to~$1$, uniformly on compact subsets of~$G$. Since $(G,M)$ has \relT, this implies that the convergence is also uniform on~$M$. However, if $\varphi_{Q,\delta}$ is sufficiently close to~$1$ on all of~$M$, then $\xi_{Q,\delta}$~is $(M,\epsilon)$-invariant. This is a contradiction.
\end{aid}

\begin{aid} \label{InvtMeasFromInvtVector}
Let $\mu' = \|f'\|^2 \, \mu$. For $E \subseteq \dual A$ and $m \in M$, we have
	\begin{align*}
	| (m_*\mu')(E) - \mu'(E)|
	&=| \mu'(m^{-1}E) - \mu'(E)|
	\\& = \left| \int_{m^{-1}E} \|f'\|^2 \, d\mu - \int_E \|f'\|^2 \, d\mu  \right|
	\\&=  \left| \int_E \| f'(m^{-1} \lambda)\|^2 \, D(m,\lambda) \, d\mu(\lambda) - \int_E \|f'\|^2 \, d\mu  \right|
	\\&=  \left| \int_E \|  \pi'(m)f' \|^2 \, d\mu - \int_E \|f'\|^2 \, d\mu  \right|
	\\&=   \int_{\dual A} \bigl( \| \pi'(m)f' \|^2 - \|f'\|^2 \bigr) \, d\mu  
	\\&=   \int_{\dual A} \bigl( \| \pi'(m)f' \| + \|f'\| \bigr)\bigl( \| \pi'(m)f' \| - \|f'\| \bigr) \, d\mu  
	\\&\le   \int_{\dual A} \bigl( \| \pi'(m)f' \| + \|f'\| \bigr)\, \| \pi'(m)f' - f'\|  \, d\mu  
	\\&\le  \bigl\| \| \pi'(m)f' \| + \|f'\| \bigr\|_2 \,   \| \pi'(m)f' - f'\|_2 
		\quad \begin{pmatrix} \text{\ H\"older} \\ \text{Inequality} \end{pmatrix}
	\\&\le  \bigl( \| \pi'(m)f' \|_2 + \|f'\|_2 \bigr) \,  \frac{\delta'}{2}
	\\&=  2 \, \frac{\delta'}{2}
	\\&= \delta'
	. \end{align*}
\end{aid}

 \begin{aid} \label{ApproxForPair}
 We provide the proof that if $(G,M)$ has \relT, then $(G,M)$ has \relT\ with approximation. (This is adapted from Jolissaint~\cite{Jolissaint}.)
 
 Let $I$ be the set of all pairs $(Q, \delta)$, such that $Q$ is a nonempty compact subset of~$G$ and $\delta > 0$. This is  a net if we specify that $(Q_1 , \delta_1) \preceq (Q_2 , \delta_2)$ if and only if $Q_1 \subseteq Q_2$ and $\delta_1 \ge \delta_2$.

If $(G,M)$ does not have \relT\ with approximation, then there exists $\epsilon > 0$, such that for every $i = (Q,\delta) \in I$, there is a unitary representation~$\pi_i$ of~$G$ with a $(Q,\delta)$-invariant unit vector~$\xi_i$ that is not $(M,\epsilon)$-invariant.

Let $\varphi_i(g) = \langle \pi_i(g) \xi_i \mid \xi_i \rangle$. Then each $\varphi_i$ is positive-definite, and $\varphi_i \to 1$ uniformly on compact sets. Since $(G,M)$ has \relT, we see from \cite[Thm.~1.1 ($2 \Rightarrow 1$)]{Cornulier} that $\varphi_i \to 1$ uniformly on~$M$. (This is Cornulier's \emph{definition} of \relT\ when $M$ is a subset.) Therefore $\sup_{m \in M} \| \pi_i(m) \xi_i - \xi_i \| \to 0$. 

This contradicts the fact that $\sup_{m \in M} \| \pi_i(m) \xi_i - \xi_i \| \ge \delta$, since $\xi_i$ is not $(M,\delta)$-invariant.
\end{aid}

\begin{aid} \label{actR/Ufinite}
Let $R$ and $U$ be the radical and nilradical of a connected Lie group~$G$. 
We provide a proof of the well-known fact that $\Aut G$ acts on $R/U$ via a finite group. 
Note that, since every element of $\Aut G$ acts on~$R$ via an element of $\Aut R$, it suffices to show that $\Aut R$ acts on $R/U$ via a finite group. 

Let $\Lie R$ be the Lie algebra of~$R$, let $\mathbf R$ be the Zariski closure of $\Ad R$ in $\mathbf{GL}(\Lie R)$, and let $\mathbf U$ be the unipotent radical of~$\mathbf R$. 
Since $\Aut \Lie R$ normalizes $\Ad R$, it also normalizes~$\mathbf R$, and therefore acts on $\mathbf R/\mathbf U$. Since $\mathbf R$ is a subgroup of $\Aut \Lie R$, and $\Aut \Lie R$ is Zariski closed, this can be viewed is the action of the algebraic group $(\Aut \Lie R)/ \mathbf U$ by conjugation on the normal subgroup $\mathbf R/\mathbf U$. However, since $\mathbf R$ is solvable, it is well known that the algebraic group $\mathbf R/\mathbf U$ is a torus \cite[\S19.1, p.~122]{Humphreys-LinAlgGrps} (that is, $\mathbf R/\mathbf U$ is connected and is isomorphic, as an algebraic group, to a group of diagonal matrices).
It is also well known that the centralizer of any torus in an algebraic group has finite index in its normalizer \cite[Cor.~16.3, p.~106]{Humphreys-LinAlgGrps}. Therefore, the action of $\Aut \Lie R$ on $\mathbf R / \mathbf U$ must be by a finite group.

Let $V$ be the kernel of the natural homomorphism $R \to \mathbf R/\mathbf U$. Then $V^\circ$ is a connected, normal subgroup of~$R$. Also, by definition, we have $\Ad_R V \subseteq \mathbf U$, so $\Ad V^\circ$ is unipotent. This implies that $V^\circ$ is nilpotent \cite[Cor.~17.5, p.~113]{Humphreys-LinAlgGrps}. So $V^\circ$ is contained in the nilradical~$U$ of~$R$, which means that $R/U$ is a quotient of $R/V^\circ$. 

The first paragraph shows that some finite-index subgroup of $\Aut R$ centralizes $R/V$. Since $R/V^\circ$ is connected, and $V/V^\circ$ is discrete, this finite-index subgroup must also centralize $R/V^\circ$. From the preceding paragraph, we conclude that this finite-index subgroup centralizes $R/U$. So $\Aut R$ acts on $R/U$ by a finite group.

\vbox{

}

\end{aid}

\begin{aid} \label{NormalLie-nonormalAid}
Suppose $M$ is a nontrivial, connected, $H$-invariant, normal subgroup of~$N$, such that the space $\hilbert^M$ of $M$-invariant vectors is nonzero. Let $N_0 = N/M$. (We may assume $M$ is closed,  since every $M$-invariant vector is also $\cl{M}$-invariant.) Since $M$ is an $H$-invariant, normal subgroup of~$N$, we know $M \triangleleft H \ltimes N$, so $\hilbert^M$ is $(H \ltimes N)$-invariant. Therefore, we obtain a representation $\pi_0$ of $H \ltimes N_0$ by restricting~$\pi$ to~$\hilbert^M$. From \cref{NearInvariantForPair}, we see that $\pi_0$ has almost-invariant vectors. 

Let $S_0$ be the closure of the image of~$S$ in~$N_0$, let $U_0$ be the nilradical of~$N_0$, and let 
	$$ (U_S)_0 = \frac{U_0}{\cl{ \vphantom{\bigl(} [S_0,U_0] \cdot (S_0 \cap U_0)}} .$$ 
We claim that \pref{NormalLie-S}, \pref{NormalLie-R}, and~\pref{NormalLie-U} hold with $N_0$, $S_0$, $U_0$, and $(U_S)_0$ in place of $N$, $S$, $U$, and~$U_S$. First of all, \pref{NormalLie-S}$_0$ is immediate from \pref{NormalLie-S}, since the image of~$S$ is dense in~$S_0$. And \pref{NormalLie-R}$_0$ is immediate from \pref{NormalLie-R}, since $N_0/(S_0U_0)$ is a quotient of $N/(SU)$. However, \pref{NormalLie-U}$_0$ is not quite immediate from \pref{NormalLie-U}, because $U_0$ may be larger than the image~$U_1$ of~$U$ in~$N_0$. However, we see from \pref{NormalLie-R} that $U_0/\cl{U_1}$ is compact, so the difference is not large enough to affect \relT.

Since $\dim N_0 = \dim N - \dim M < \dim N$, we conclude by induction on the dimension that $(H \ltimes N_0, N_0)$ has \relT, so there are nonzero $\pi_0(N_0)$-invariant vectors. Since $\pi_0$ is a restriction of~$\pi$, these vectors are $\pi(N)$-invariant.
\end{aid}

\begin{aid} \label{NormalLie-Snormal}
Let $\widetilde N$ be the universal cover of~$N$, let $\widetilde S$ be the product of the noncompact, simple factors of a Levi subgroup of~$\widetilde N$, and let $\widetilde U$ be the nilradical of~$\widetilde N$. It is well known that $[\widetilde S, \widetilde N] = \widetilde S[\widetilde S,\widetilde U]$. However, if we assume $\widetilde S$ has been chosen so that its image in~$N$ is dense in~$S$, then we know that $\widetilde S$ centralizes~$\widetilde U$, so $[\widetilde S,\widetilde U]$ is trivial. Therefore $[\widetilde S, \widetilde N] = \widetilde S$, so $\widetilde S \triangleleft \widetilde N$. Since all Levi subgroups are conjugate, this implies that $\widetilde S$ is characteristic in~$\widetilde N$, so $\widetilde S \triangleleft H \ltimes \widetilde N$. By applying the covering map $\widetilde N \to N$, we conclude that $S \triangleleft H \ltimes N$.
\end{aid}

\begin{aid} \label{N1MaxCpct}
Since $N$ is connected, \cref{open/cpct=Lie} tells us that $N$ has a compact, normal subgroup~$C$, such that $N/C$ is a (connected, solvable) Lie group. Then Lie's Theorem in the structure theory of connected, solvable Lie groups tells us that $N^{(1)}C/C$ is nilpotent 
\cite[Cor.~C, p.~16]{Humphreys}, and therefore has a unique maximal compact subgroup~$C_1/C$ \csee{NilpHasCpct}. 
Then $C_1 \cap N^{(1)}$ is the unique maximal compact subgroup of~$N^{(1)}$. 

Also, we have $N^{(1)}/(C_1 \cap N^{(1)}) \cong N^{(1)} C_1/C_1$. 
Since $C_1$ contains~$C$, this is isomorphic to a quotient of $N^{(1)}C/C$, which is nilpotent. So $N^{(1)}/(C_1 \cap N^{(1)})$ is nilpotent.

\vbox{

}

\end{aid}


\vfill\vfill
\end{appendix}

\end{document}